\documentclass[pdflatex,sn-mathphys-num]{sn-jnl}
\usepackage{enumitem}

\usepackage{graphicx}%
\usepackage{multirow}%
\usepackage{amsmath,amssymb,amsfonts}%
\usepackage{amsthm}%
\usepackage{mathrsfs}%
\usepackage[title]{appendix}%
\usepackage{xcolor}%
\usepackage{textcomp}%
\usepackage{manyfoot}%
\usepackage{booktabs}%
\usepackage{algorithm}%
\usepackage{algorithmicx}%
\usepackage{algpseudocode}%
\usepackage{listings}%

\usepackage{mathptmx}      
\usepackage{latexsym}

\usepackage{comment}

\newtheorem{proposition}{Proposition}[section]
\newtheorem{theorem}[proposition]{Theorem}
\newtheorem{definition}[proposition]{Definition}
\newtheorem{corollary}[proposition]{Corollary}
\newtheorem{lemma}[proposition]{Lemma}
\newtheorem{remark}[proposition]{Remark}

\definecolor{verde}{rgb}{0.0, 0.5, 0.0}
\definecolor{auburn}{rgb}{1, 0.5, 0.15}
\definecolor{bordo}{rgb}{.65, 0.05, 0.25}

\newcommand{\He}{\mathcal{H}}

\newcommand{\Se}{\mathcal{S}}

\begin{document}

\title[Optional Stopping for Superhedging Supermartingales]{Optional Stopping for Superhedging Supermartingales}



\author[1]{\fnm{C.} \sur{Bender}}

\author*[2]{\fnm{S. E.} \sur{Ferrando}}\email{ferrando@torontomu.ca}



\affil[1]{\orgdiv{Department of Mathematics},
\orgname{Saarland University},
\orgaddress{\street{Campus E 2 4},
\city{Saarbrücken},
\postcode{66123},
\country{Germany}}}

\affil[2]{\orgdiv{Department of Mathematics},
\orgname{Toronto Metropolitan University},
\orgaddress{\street{350 Victoria St.},
\city{Toronto},
\postcode{M5B 2K3},
\state{Ontario},
\country{Canada}}}


\date{Received: date / Accepted: date}




\abstract{
Superhedging supermartingales, introduced in \cite{bender3}, are non-probabilistic processes defined via subadditive outer integrals that carry a purely financial interpretation in terms of superhedging cost. Building on the Leinert–K\"onig theory of non-lattice integration, the present paper establishes several results that are classical in probability theory but whose non-probabilistic proofs require fundamentally new arguments: 
(i) a tower inequality for the conditional outer integral 
$\overline{\sigma}_j$ applied at stopping times, reducing to equality when the integrand is conditionally integrable; (ii) three versions of Doob's optional stopping theorem, organised by the class of supermartingale and the range of the stopping times; and (iii) Dubins' upcrossing inequality in both finite- and infinite-time horizons. A key structural result, property $(K)$-a.e., identifies conditions under which the two superhedging operators $\overline{\sigma}_j$ and $\overline{I}_j$ coincide on non-negative functions, extending the scope of all preceding results to the positive operator $\overline{I}_j$. None of the proofs invoke classical measure-theoretic tools; in particular, (classical) integrability and measurability are not assumed. The analogues of classical stochastic results acquire a purely financial interpretation and, in this way, gain depth and generality by providing a context that is independent of any a priori probabilistic structure.



}



\maketitle


\section{Introduction}

This paper develops a trajectorial analogue of several cornerstone results of classical martingale theory: Doob's optional stopping theorem, the tower property for conditional expectations, and Dubins' upcrossing inequality, in a setting where no probability measure is given. Instead, the underlying objects are a set $\mathcal{S}$ of price scenarios and a pair of subadditive operators $\overline{\sigma}_j$ and $\overline{I}_j$, each with a direct superhedging interpretation.  Our main contribution is to show that these classical results survive the transition to the non-probabilistic setting, but require entirely new proof techniques.

The framework, developed in detail in \cite{bender1,bender3}, extends the non-lattice integration theory of Leinert \cite{leinert} and König \cite{konig} to a conditional financial setting while maintaining a formal connection with classical probabilistic constructions. In general, the associated integral operators need not arise from measures (see \cite{bender1,bender3}). Nevertheless, the framework supports notions such as supermartingale processes that parallel classical probabilistic concepts while admitting a purely financial interpretation. Section 4 of \cite{bender3} compares this approach with the closely related framework of game-theoretic probability developed in \cite{shafer}. The present paper builds on the definitions and basic framework introduced in \cite{bender3} and develops several classical probabilistic results in this non-probabilistic setting. The proofs are obtained directly within the framework and do not rely on classical probabilistic arguments. Our goal is to further illustrate that a substantial portion of probabilistic methodology can be formulated in a purely hedging-based setting.

We comment on the nature of our main results: a tower inequality for conditional outer integrals, several versions of Doob’s optional sampling theorem, and Dubins’ upcrossing inequality. In classical measure-based probability, these theorems are formulated via expectation operators and require strict measurability and integrability hypotheses—though one might speculate that they could be reformulated using outer expectations, as in \cite{VdVW}, to weaken the latter. In contrast, our results require neither integrability nor measurability because the domains of our outer integral operators are completely unconstrained. Furthermore, while classical proofs rely heavily on the linearity of conditional expectation, Fatou's lemma, and the monotone convergence theorem, none of these tools are available for our subadditive operators. Instead, each proof must be constructed directly from the definition of superhedging and the supermartingale decomposition of \cite{bender3}[Theorem 6.1]. The fact that these arguments nevertheless proceed naturally demonstrates that these classical principles remain valid when framed purely in terms of superhedging.

We remark that we are in a discrete, infinite time, setting and that the base space $\mathcal{S}$ has no restrictions besides some no-arbitrage requirements that need to hold so that our main hypothesis, namely the $(L)-a.e.$ property, holds. The said property allows us access a supermartingale decomposition theorem from \cite{bender3} (here Theorem \ref{thm:decomposition}).
We also note that our outer conditional integral operators are defined everywhere, this is in sharp contrast  to the case of Kolmogorov's conditional expectations that are defined a.e. Relatedly, fundamental null sets in our setting have a superhedging interpretation, a fact that clarifies several of our proof techniques. We also make a fundamental use of stopping times that are naturally defined in our non-probabilistic setting.

The paper is organized as follows: Section \ref{backgroundNotions} presents the basic definitions and the main results
from \cite{bender3} that we rely upon as background material. Section \ref{towerProperty} develops the fundamental tower property in the form of an inequality, we remark that this result requires no hypothesis. An analogue of the classical tower (equality) property is obtained under some integrability hypotheses. Section \ref{optionalStopping} presents three versions of Doob's optional stopping results, mostly organized according to a natural class of supermartingales and hypthesis required on bounded versus infinite time.   Section \ref{dubinsInequality} develops a finite time and an infinite time versions of Dubins' inequality for upcrossings. The infinite time version requires additional hypothesis to hold.  Section \ref{propertyKae} establishes the equality $\overline{\sigma}_j= \overline{I}_j$ on nonnegative functions, this technical result makes several results in the paper available, under the required hypotheses, for both classes of operators. 

\section{Basic setting and definitions}  \label{backgroundNotions}

This section introduces material from \cite{bender3} that is technically needed for the developments of the present paper. We omit any justification, remarks or additional material that is not essentially required and refer the reader to \cite{bender3} for explanations and details.

\subsection{Trajectorial Setting} 
{\bf Trajectory set} {\cite[Definition 1]{ferrando}}
Given a real number $s_0$, a \emph{trajectory set}, denoted by $\mathcal{S}= \mathcal{S}(s_0)$, is a subset of $\mathcal{S}_{\infty}(s_0) =\{S=(S_i)_{i\in \mathbb{N}\cup\{0\}}: S_i\in \mathbb{R},~ S_0=s_0\}.$ We make fundamental use of the following \emph{conditional spaces}; for $S \in \Se$ and $j \geq 0$  set:
$$
\Se_{(S,j)}\equiv\{\tilde{S} \in \Se: \tilde{S}_i= S_i, ~~ 0 \le i \le j\},
$$
the notation $(S,j)$, henceforth referred as a {\it node},  will be used as a shorthand for $\Se_{(S,j)}$. {\it Local} properties are relative to a given node.

\noindent
{\bf Conditional portfolio set} (\cite{bender3}[Definition 2.3] 
For any fixed $S\in \Se$ and $j\ge 0$, $\He_{(S,j)}$ will be the set of all sequences of functions $H = (H_i)_{i \geq j}$, where $H_i: \Se_{(S,j)} \rightarrow \mathbb{R}$ are non-anticipative in the sense: for all $\tilde{S},\hat{S}\in\Se_{(S,j)}$ such that $\tilde{S}_k = \hat{S}_k$ for $j \le k\le i$, then $H_i(\tilde{S}) = H_i(\hat{S})$ (i.e., $H_i(\tilde{S}) = H_i(\tilde{S}_0,\ldots, \tilde{S}_i)$). We introduce the shorthand notation $\mathcal{H}= \He_{(S,0)}$.
No measurability assumptions are imposed on the portfolio maps $H_i$. This is consistent with the use of subadditive outer integral operators in the present framework. Further discussion and connections with outer expectation constructions and game-theoretic probability are provided in \cite{bender3}.

For a node $(S,j)$, $H\in\He_{(S,j)}$, $V\in \mathbb{R}$ and $n\ge j$ we define the portfolio wealth process $\Pi_{j,n}^{V, H}: \mathcal{S}_{(S,j)} \rightarrow \mathbb{R}$, as:
$$
\Pi_{j,n}^{V, H}(\tilde{S}) \equiv V+\sum_{i=j}^{n-1}H_i(\tilde{S})~ \Delta_i \tilde{S},\quad
\mbox{where}\;\: \Delta_i \tilde{S} = \tilde{S}_{i+1}- \tilde{S}_i,~~i\ge j, ~~\tilde{S} \in \mathcal{S}_{(S,j)}.
$$
Notice that $V$ is assumed to be constant on $\mathcal{S}_{(S,j)}$ and so its value could change with $S$, i.e., $V= V(S)$ (depending on the past stock price evolution up to time $j$).  
 
 In the sequel, being $\mathcal{A}$ a set of real valued functions, $\mathcal{A}^+$ will denote the set of its non-negative elements.

\noindent
{\bf  Elementary vector spaces}, for a fixed node $(S,j)$  set
$\mathcal{E}_{(S,j)}= \{f = \Pi_{j, n_f}^{V, H}: H \in \He_{(S,j)},~~V \in \mathbb{R}~~~\mbox{and}~~~n_f \in \mathbb{N}\}.$
%
Observe that $\mathcal{E}_{(S,j)}$ is a real vector space. Its elements are called \emph{elementary functions}.  Let also define
$\mathcal{E}_j = \{f:\Se\rightarrow \mathbb{R}: f|_{\Se_{(S,j)}}\in\mathcal{E}_{(S,j)} \;\; \forall S\in\Se \}$,
where the notation $f|_{\Se_{(S,j)}}$ means that the global domain of $f$, namely $\mathcal{S}$, is being restricted to the subset $\Se_{(S,j)}$. We note in passing some abuse of notation as the same symbol $\Pi_{j, n_f}^{V, H}$ is 
used to denote elements from $\mathcal{E}_{(S,j)}$ and $\mathcal{E}_j$ (in particular the implicit dependence on $S$ is not made explicit in the
case when $\Pi_{j, n_f}^{V, H} \in \mathcal{E}_{(S,j)}$). More details on global versus local portfolios are provided in \cite{bender1}.

\subsection{Fundamental Operators and Almost Everywhere Notions}\label{a.e. section}

The operators introduced in this subsection provide the conditional superhedging and conditional almost everywhere structures underlying the trajectorial supermartingale framework developed in \cite{bender3}. As the optional stopping and convergence results established later will show, this framework is sufficiently robust to support several classical martingale-type constructions in a purely trajectorial setting.

Let $Q$ denote the set of all functions from $\mathcal{S}$ to $[-\infty,\infty]$ and let $P\subseteq Q$ be the subset of non-negative functions. The following conventions are in effect:
\[
0\,\infty=0,\qquad \infty+(-\infty)=\infty,\qquad
u-v\equiv u+(-v), \quad u,v\in[-\infty,\infty],
\]
and $\inf\emptyset=\infty$ (unless indicated otherwise).

We define next the \emph{conditional norm operator} \; $\overline{I}_j : P \rightarrow \mathcal{E}^+_j$, which is a conditional extension of the operator $\overline I$ defined in \cite{ferrando}, and is used to define null sets. 

\begin{definition}\label{Iup_definition}
For a given node $(S, j)$ and a general $f \in P$ define
\begin{equation} \nonumber
\overline{I}_j f (S)\equiv  \inf \left\{\sum_{m \geq 1} V^m: ~~f \leq  \sum_{m \geq 1} \liminf_{n\to \infty}~~\Pi_{j, n}^{V^m, H^m}\;\;\;\mbox{on}\;\; \Se_{(S,j)},\; \Pi_{j, n}^{V^m, H^m}\in \mathcal{E}_{(S,j)}^+  ~~\forall ~~n \geq j \right\}.
\end{equation}
We will use the notation $\overline{I}f \equiv \overline{I}_0f$. We also set, for a general $f \in Q$:
\begin{equation} \nonumber
||f||_j(S) \equiv \overline{I}_j|f|(S)~~\mbox{and}~~||f|| \equiv ||f||_0(S).
\end{equation}
\end{definition}
Notice that $\overline{I}_j f (S)= \overline{I}_j f (S_0, \ldots, S_j)$, i.e. $\overline{I}_j f (\cdot)$ is constant on $\mathcal{S}_{(S,j)}$. Moreover, the positivity constraint imposed on the portfolio values in Definition \ref{Iup_definition} guarantees that $\sum_{m\geq 1}V^m\geq 0$ and consequently $\overline{I}_j f\geq 0$. This positivity structure is also fundamental for the countable subadditivity properties of $\overline{I}_j$ used throughout the paper. We set $||0||_j=0$ and refer to $||\cdot||_j(S)$ as a {\it conditional norm}.

We next introduce the conditional almost everywhere notions associated with the operator $\overline I_j$ (\cite{bender3}[Definition 2.7]). Given a node $(S,j)$, a function $g\in Q$ is a \emph{conditionally null function at} $(S,j)$  if:\[\|g\|_j(S)=0.\]
 A subset $E\subset\Se$ is a \emph{conditionally null set at $(S,j)$} if $\|\mathbf{1}_E\|_j(S)=0$. A property is said to hold conditionally a.e. at $(S,j)$
(or equivalently: the property holds ``a.e. on $\mathcal{S}_{(S,j)}$") if the subset of  $\Se_{(S,j)}$ where it does not hold
is a conditionally null set at $(S,j)$. In particular, the latter definition applies to $g=f$ a.e. on $\mathcal{S}_{(S,j)}$, which also will be noted with $g\doteq f$ when $j=0$.

Notice that when $j=0$, the previous notions do not depend on $S$ and we apply the abbreviation ``a.e.'' for ``a.e. at $(S,0)$''. Moreover, $E\subset\Se$ is called a \emph{null set} and $g$ is called a \emph{null function}, if $\|\mathbf{1}_E\|=0$ and $\|g\|=0$, respectively.
We refer to \cite{bender3}[propositions 2.8 and 2.9] for basic properties of $\overline{I}$ and null sets. 

All appearing equalities and inequalities are valid for all points in the spaces where the functions are defined unless qualified by an explicit a.e. The notation
$f \dot{=} g$ ($f \dot{\leq} g$) is also used for the equality (inequality) being valid only a.e.

\vspace{.1in}
We introduce next the operator $\overline{\sigma}_j:Q\rightarrow\mathcal E_j$, referred to as the \emph{conditional superhedging operator} (or conditional outer integral). This operator provides the basic conditional structure underlying the trajectorial supermartingales and stopping-time results developed later.

\noindent
{\bf Conditional Outer Integral,} or a node $(S,j)$ and a general $f \in Q$,
\begin{equation} \nonumber
\overline{\sigma}_j f(S) \equiv  \inf \left\{\sum_{m \geq 0}V^m: ~~f \leq  \sum_{m \geq 0} f_m \;\;\mbox{on}\; \Se_{(S,j)}\right\},
\end{equation}
where $f_0=\Pi_{j, n_0}^{V^0,H^0}\!\!\in \mathcal{E}_{(S,j)}; \; \mbox{for}\; m \geq 1, \; f_m= \liminf_{n\to\infty}~~\Pi_{j, n}^{V^m, H^m},~~\mbox{and}\;\;\Pi_{j, n}^{V^m, H^m}\!\!\in \mathcal{E}_{(S,j)}^+~~\forall ~~~n \geq j$.
Define also $\underline{\sigma}_j f(S) \equiv -\overline{\sigma}_j(-f) (S)$.
We will use the notation $\overline{\sigma}f \equiv \overline{\sigma}_0 f$.

The following definitions will help us to provide a more uniform presentation of the results.
We will set, for $S \in \mathcal{S}$,
\begin{equation} \label{atTheEndOfTime}
\mathcal{S}_{(S, \infty)} \equiv \{S\}~~~\mbox{and}~~\overline{\sigma}_{\infty} f (S)=  f(S)~\mbox{for}~ \in Q,
\end{equation}
we also set $\underline{\sigma}_{\infty} f(S)= - \overline{\sigma}_{\infty} (-f)(S)
= f(S)$.

\noindent The quantity $\overline{\sigma}_j f(S)$ depends on $S$ only through the coordinates $(S_0,\ldots,S_j)$. Also, for notational convenience, the term $f_0$ may be written in the same form as the remaining terms by setting $f_0=\liminf_{n\to\infty}\Pi_{j,n}^{V^0,H^0}$,
with $H_i^0\equiv 0$ for $i\ge n_0$.

The need to rely on $\liminf$ in the definitions of $\overline{I}_j$ and $\overline{\sigma}_j$ appears so as to handle particular null sets that emerge in infinite time such as in Doob's poitnwise convergence theorem (\cite{bender3}[Theorem 7.1]).

\subsection{Property $(L)$}
\vspace{.1in}
The property $(L_{(S,j)})$, introduced in the following definition,
generalizes a non-conditional version from  \cite{leinert} and will be called (conditional) \emph{continuity from below}.
Property $(L_{(S,j)})$ will be key to obtain several of the results in the paper, in particular it will imply that $\overline{\sigma}_j$ preserves the property of non-negativity.

\begin{definition}[Property $(L_{(S,j)})$] \label{lPropertyWithLimInf}
For a fixed node $(S,j)$, $~~~~f = \Pi^{V, H}_{j, n_f}\in \mathcal{E}_{(S, j)}$
and $f_m = \liminf_{n \rightarrow \infty} \Pi^{V^m, H^m}_{j, n}$ with $\Pi^{V^m, H^m}_{j, n} \in \mathcal{E}_{(S, j)}^+~\mbox{for all}~ n \geq j~\mbox{and}~ m \geq 1$, ~define property $(L_{(S,j)})$ by
\begin{equation} \nonumber 
(L_{(S,j)}):\quad f \leq \sum_{m \geq 1} f_m\;\;\mbox{on}\; \Se_{(S,j)}
\implies V \leq \sum_{m \geq 1}  V^m.
\end{equation}
If $(L_{(S,j)})$ holds for $S\in\Se\;\; \overline{I}-a.e.$, it will be written $(L_j)$ holds $a.e.$ (or $\overline{I}-a.e.$ in case of possible ambiguity).
\end{definition}


Since $(L_{(S,0)})$ does not depend on $S$, it will be denoted by $(L)$. If $(L_{(S,j)})$ holds, it follows that $\overline{\sigma}_j f(S) = V$ for $f= \Pi^{V, H}_{j,n} \in \mathcal{E}_{(S, j)}$ (see \cite[Proposition 3.3, item 4]{bender3}). 
In particular,  \cite[Proposition 3.3]{bender3} shows that property $(L_{(S,j)})$
is equivalent to $0 \leq \overline{\sigma}_j 0 (S)$ and so to $0 = \overline{\sigma}_j 0(S)$ which in turn implies $\underline{\sigma}_j f(S)  \leq \overline{\sigma}_jf(S)$
for any $f:\mathcal{S} \rightarrow [- \infty, \infty]$ (i.e. $f \in Q$). 

Note that the conditional outer integral $\overline{\sigma}_j f(S)$ and the conditional inner integral  $\underline{\sigma}_j f(S)$ are defined at any node $(S,j)$ and for any function $f\in Q$.  However, $\overline{\sigma}_j f(S)=-\infty$ and  $\underline{\sigma}_j f(S)=+\infty$ for any bounded function $f$, if  $(L_{(S,j)})$ fails at node $(S,j)$. Therefore we shall impose the following assumption for several results.

\vspace{.1in}
\noindent
\begin{definition}[Assumption $(L)$-a.e.]: \label{assumptionL-ae}
 we will write $(L)$-a.e. whenever the following two properties hold:
 \begin{itemize}
\item[i)] $(L)$ (i.e. $(L_{(S,0)})$) holds,

\item[ii)] 
\begin{equation}\label{eq:N^L}
\mathcal{N}^{(L)} \equiv \{S\in\Se: \exists j\ge 0 \; s.t. \; (L_{(S,j)})\; \mbox{fails}\}~~\mbox{is a null set}. 
\end{equation}
\end{itemize}
\end{definition}

\vspace{.1in}
\noindent
The validity of $(L_{(S,j)})$ and $(L)-a.e.$ are established, under appropriate conditions, in \cite{bender3}.

\begin{remark} \label{L0ForNonTrivialityOfComplement}
Therefore, property $(L)$-a.e. implies that the property
$[(L_{(S, j)})~~\mbox{holds for all}~~~j \geq 0]$ is valid $\overline{I}$-a.e. It could be that $\mathcal{S} \setminus \mathcal{N}^{(L)} = \emptyset$; a general assumption to obstruct this from happening is to assume that $(L_{(S,0)})$ holds. In fact if $(L_{(S,0)})$ is valid, one can check that $\overline{I}({\bf 1}_{\mathcal{S}})=1$ and so $1= \overline{I}({\bf 1}_{\mathcal{S}}) \leq
\overline{I}({\bf 1}_{\mathcal{S} \setminus \mathcal{M}})+ \overline{I}({\bf 1}_{\mathcal{M}}) \leq  \overline{I}({\bf 1}_{\mathcal{S} \setminus \mathcal{M}}) \leq 1$ where $\mathcal{M}$ is any arbitrary null set and then $\overline{I}({\bf 1}_{\mathcal{S} \setminus \mathcal{N}^{(L)}})=1$ whenever $(L)$-a.e. and $(L_{(S,0)})$ both hold.
\end{remark}

\begin{definition}[Conditional Integrable Functions] \label{integrable}
$f \in Q$
is called a \emph{conditional integrable} at $j$ function if it satisfies:
\begin{equation} \nonumber
\underline{\sigma}_j f = \overline{\sigma}_j f,~~ \overline{I}-a.e.
\end{equation}
In particular, $f \in Q$ is integrable if
 $\underline{\sigma}_0 f= \overline{\sigma}_0 f$.
\end{definition}
Integrable functions will play a side role in our work that deals mostly with inequalities, integrability appears when requiring equality in our tower property and when specializing supermartingales to martingales in our trajectorial setting.

\vspace{.15in}
We next introduce trajectorial supermartingales, submartingales, and martingales (following \cite{bender3}).
Throughout, $T \in \mathbb{N}\cup\{\infty\}$ denotes the time horizon; whenever not explicitly indicated, it means a given result holds for either finite time i.e. $T< \infty$ or infinite time $T= \infty$. A sequence $\{f_j\}_{0 \leq j < T+1}$, $f_j \in Q$, is said to be non-anticipative if for each $S\in\Se$ and $0 \leq j < T+1$, $f_j$ is constant on $\Se_{(S,j)}$. The index $T+1$, either finite or infinite, may need to be available as a possible time for some of our stopping times. This is the case for upcrossing times that are defined thorough an infimum which, in turn, may be assigned the index $T+1$ whenever the defining conditions amount to empty sets. In those instances, we will define $f_{T+1}$ appropriately.  The case of finite sequences $f_j$, $0 \leq j \leq T < \infty$, whenever convenient, will be embedded in an infinite time case via $f_j \equiv f_T$ for all $T+1 \leq j < \infty$. We also write $\{f_j\}_{j \geq 0}$ whenever $0 \leq j < \infty$ that corresponds to the case $T= \infty$.

\begin{definition} \label{superMartingale}
Consider a sequence $\{f_n\}= \{f_n\}_{0 \leq n < T+1}$ of non-anticipative functions $f_n:\mathcal{S} \rightarrow [- \infty, \infty]$, $n\ge0$.

\vspace{.1in}\noindent $\{f_n\}$ is a \emph{supermartingale} sequence if
\begin{equation}  \nonumber 
\overline{\sigma}_j f_{j+1} \leq f_j~~a.e.\;\; 0 \leq j < T,
\end{equation}
$\{f_n\}$ is a \emph{submartingale} if
\begin{equation} \nonumber 
f_j \leq \underline{\sigma}_j f_{j+1}~~a.e.\;\;0 \leq j < T,
\end{equation}
$\{f_n\}$ is a \emph{martingale} if
\begin{equation}\nonumber 
\underline{\sigma}_j f_{j+1}= \overline{\sigma}_j f_{j+1}= f_j~~a.e.\;\; 0 \leq j < T.
\end{equation}
\end{definition}

The following result, \cite[Theorem 6.1]{bender3},  will be a main tool used several times in the paper; we present it here for the reader's convenience.
\begin{theorem}[Supermartingale decomposition] \label{thm:decomposition} 
Under Assumption $(L)$-a.e.,	let $(f_j)_{j\ge 0}$ be a sequence of
	non-anticipative real-valued functions. Then, the following assertions are equivalent:
	\begin{itemize}
		\item [(i)] $(f_j)_{j\ge 0}$ is a supermartingale.
		\vspace{.05in}
		\item [(ii)] For every sequence $(\delta_j)_{j\ge 0 }$ of positive real numbers there are sequences $(H_j)_{j\ge 0}$ and $(A_j)_{j\ge 0}$, of non-anticipative real-valued functions defined on 
		$\mathcal{S}$, such that $(A_j)_{j\ge 0}$ is nondecreasing, $A_0 = 0$, and
		\begin{equation} \label{supermartingaleDecomposition}
		f_i(S)=f_0+\sum_{j=0}^{i-1}H_j(S)\Delta_jS-A_i(S)+\sum_{j=0}^{i-1}\delta_j,
		\end{equation}
		for every $S \in \Se \setminus N_f$ and $i\ge 0$. Here $N_f$ is an $\overline{I}$-null set independent of  $(\delta_j)_{j\ge 0 }$.
	\end{itemize}
\end{theorem}

Property $(P)$, see below, will be used in conjunction with Theorem \ref{thm:decomposition} by relying on \cite[Lemma 7.3]{bender3}. The property provides a sufficient condition to establish
the positivity required of elementary portfolios in settings where time $T=\infty$. For introducing $(P)$ below, we recall from \cite{bender3} that a node $(S,j)$ is said to be an \emph{up-down node}, if there are $S^1,S^2\in \Se_{(S,j)}$ such that $S^1_{j+1}<S_j<S^2_{j+1}$ (see also Definition \ref{typesOfNodes}).

\begin{definition}[(P)]  \label{def:Positivity}
The following  property is relative to $\mathcal{S}$.
\begin{itemize}[labelwidth=!, leftmargin=*, align=left]

\item[(P)]   Whenever $(S,k)$ is an up-down node such that $(L_{(S,k)})$ holds and $(L_{(S,k+1)})$ fails, then there are $S^1, S^2\in \Se_{(S,k)}$ such that $S^1_{k+1}> S_{k+1}> S^2_{k+1}$ and $(L_{(S^n, k+1)})$ holds  for $n =1,2$.
\end{itemize}
\end{definition}

\section{Tower Property for Superhedging Operators}  \label{towerProperty}

We first collect some properties of stopping times; the main result of the section is a version of the tower property for the conditional outer integrals at stopping times. Due to the subadditivity of $\overline{\sigma}_j$, the result holds in the form of an inequality in our setting. The full analogue version of the classical tower property is obtained when we apply $\overline{\sigma}_j$ to conditionally integrable functions
(as per Definition~\ref{integrable}).

Trajectory based stopping times as contrasted to classical developments based on filtrations (but both approaches are closely related, see \cite{shiryaev}).

\vspace{.05in}
\begin{definition}(Stopping Time \cite[Definition 5.3]{bender3})  \label{stoppingTimeDefinition}
Given a Trajectory Space $\Se$, a \emph{trajectory based stopping time} (or \emph{stopping time} for short) is a function $\tau:\Se\rightarrow \mathbb{N}\cup\{\infty\}$ such that:
\begin{equation*}
    \text{for any}\quad S,S'\in\Se\quad\text{if}\quad S_k=S'_k\quad\text{for}\quad 0\leq k\leq \tau(S),\quad\text{then}\quad \tau(S)=\tau(S').
\end{equation*}
\end{definition}

 If there exists one trajectory $S\in\Se $ such that $\tau(S)=0, $ then for all $\tilde{S}\in\Se,$ we have $\tilde{S}_0=S_0=s_0$ and so $\tau(\tilde{S})=0$ and $\{\tau=0\}=\Se.$ Therefore, we have either: $\{\tau=0\}=\Se$ or $\{\tau=0\}=\phi$.

Notice that
\begin{equation} \nonumber
\mbox{if}~~S\in\{\tau=m\}~~\mbox{then}~~~\Se_{(S,m)}\subseteq\{\tau=m\}.
\end{equation}


\vspace{.1in}
We stress that the next Lemmas \ref{constantCharactFunctions}--\ref{st-portfolio} and Theorem \ref{mainDirectionOfTowerProperty s-t} are valid without imposing the assumption $(L)$-a.e.

\vspace{.1in}
\begin{lemma}\label{constant1}  \label{constantCharactFunctions}
 Let $\tau, \kappa$ be stopping times, then, the functions ${\bf 1}_{\{\tau=\kappa\}}, {\bf 1}_{\{\tau\leq \kappa\}},  {\bf 1}_{\{\tau< \kappa\}},  {\bf 1}_{\{\tau\geq \kappa \}}$, ${\bf 1}_{\{\tau> \kappa \}}$ are constant on $\Se_{(S,\kappa(S))}$. 
\end{lemma}
\begin{proof}
Let $\tilde{S}\in \mathcal{S}_{(S,\kappa(S))}$ i.e. $\tilde{S}_i = S_i$, for $0 \leq i \leq \kappa(S)$, so $\kappa(\tilde{S})=\kappa(S)$.
If $\tau(S) = \kappa(S)$, then $\tilde{S}_i = S_i$ for $0 \leq i \leq \tau(S)$, thus $\tau(\tilde{S})=\tau(S)= \kappa(S)=\kappa(\tilde{S})$. While if $\tau(S)\neq\kappa(S)$, it follows that $\tau(\tilde{S})\neq\kappa(\tilde{S})$ (otherwise,  $\tau(\tilde{S})=\kappa(\tilde{S})=\kappa(S)$) and so $\tau(S)=\tau(\tilde{S})=\kappa(\tilde{S})=\kappa(S)).$ Thus ${\bf 1}_{\{\tau=\kappa\}}$ is constant on $\Se_{(S,\kappa(S))}$.

Since any function is constant on $\Se_{(S,\infty)}=\{S\}$, it is enough to consider $\kappa(S)<\infty$ for the other cases.
If $\tau(S)<\kappa(S)+1$ then $\tau(S)\le\kappa(S)$ thus, as before, $\tau(\tilde{S})=\tau(S)<\kappa(S)+1=\kappa(\tilde{S})+1$. While if $\tau(S)\ge\kappa(S)+1$, then also $\tau(\tilde{S})\ge\kappa(\tilde{S})+1$ (otherwise, $\tau(\tilde{S})< \kappa(\tilde{S})+1$
and so $\tau(S)=\tau(\tilde{S})<\kappa(\tilde{S})+1=\kappa(S))+1$). Thus ${\bf 1}_{\{\tau< \kappa+1\}}={\bf 1}_{\{\tau\leq \kappa\}}$ is constant on $\Se_{(S,\kappa(S))}$. Since ${\bf 1}_{\{\tau> \kappa\}}={\bf 1}-{\bf 1}_{\{\tau\leq \kappa\}}$, $ {\bf 1}_{\{\tau< \kappa\}}= {\bf 1}_{\{\tau\leq  \kappa\}}-  {\bf 1}_{\{\tau=  \kappa\}}$, and ${\bf 1}_{\{\tau\geq \kappa\}}={\bf 1}-{\bf 1}_{\{\tau< \kappa\}}$, the proof is complete.
\end{proof}

\begin{lemma} \label{takingOutCharacteristicFunctions}
Let $\tau$ be a stopping time, $f \in Q$ and fix $T \in \mathbb{N}\cup \{\infty\}$.
The following holds at all $S \in \mathcal{S}$:
\begin{itemize}
\item[i)]
\begin{equation} \nonumber
\overline{\sigma}_{\tau(S)}({\bf 1}_{\{\tau < T\}}~f)(S) \leq {\bf 1}_{\{\tau < T\}}(S)~\overline{\sigma}_{\tau(S)}f(S).
\end{equation}
\item[ii)]
\begin{equation} \nonumber
\overline{\sigma}_{\tau(S)}({\bf 1}_{\{\tau = T\}}~f)(S) \leq {\bf 1}_{\{\tau = T\}}(S)~\overline{\sigma}_{\tau(S)}f(S).
\end{equation}
\end{itemize}
Moreover, the above inequalities are actually equalities if
$\overline{\sigma}_{\tau(S)} 0(S) =0$
(which, as we have mentioned, it is equivalent to the fact that $(L_{(S, \tau(S))})$ holds).
\end{lemma}
\begin{proof}
We only provide the proof of item $i)$, the proof of $ii)$ is similar.

Consider $S$ satisfying $\tau(S) <T$, as $\tau$ is a stopping time, it follows that
$\tau(\hat{S}) <T$ for all $\hat{S} \in \mathcal{S}_{(S, \tau(S))}$. Therefore
${\bf 1}_{\{\tau < T\}}~f =f$ on $\mathcal{S}_{(S, \tau(S))}$ and so
$\overline{\sigma}_{\tau(S)}({\bf 1}_{\{\tau < T\}}~f)(S)= \overline{\sigma}_{\tau(S)}f(S) = {\bf 1}_{\{\tau < T\}}(S)~\overline{\sigma}_{\tau(S)}f(S)$. So this case
gives an actual equality in $i)$. Assume now that $\tau(S) \geq T$, as
argued above we obtain $\tau(\hat{S}) = \tau(S) \geq T$ for all $\hat{S} \in \mathcal{S}_{(S, \tau(S))}$. Therefore the left hand side of $i)$ evaluates to
$\overline{\sigma}_{\tau(S)}(0~f)(S) \leq 0 =  {\bf 1}_{\{\tau < T\}}(S)~\overline{\sigma}_{\tau(S)}f(S)$. From this computation we also see that equality holds in  this case as well if
$\overline{\sigma}_{\tau(S)} 0(S) =0$.
\end{proof}

The next lemma is \cite[Lemma 5.4]{bender3} and displayed here for convenience.
\begin{lemma}\label{st-portfolio} Let $\tau$ a stopping time and $H^k=(H^k_i)_{i\ge 0},\;k=1,2$ sequences of non-anticipative functions. For $S\in\Se, j\ge 0$ define the following functions on $\mathcal{S}_{(S,j)}$:
\[
H^{\tau}_i(\tilde{S})=\left\{
\begin{array}{lll}
H^1_i(\tilde{S}) & \mbox{if} & j\le i < \tau(\tilde{S})\\
H^2_i(\tilde{S}) & \mbox{if} & \tau(\tilde{S})\le i.
\end{array}
\right.
\]
Then $H^{\tau}=(H^{\tau}_i)_{i\ge 0}$ is a sequence of non-anticipative functions.
\end{lemma}

\vspace{.1in}
\begin{theorem} [Tower Property I] \label{mainDirectionOfTowerProperty s-t}
Let  $S$ be an arbitrary element of $\mathcal{S}$ and  $\rho \le \tau$ be two stopping times; also let $f \in Q$. Then, 
\begin{equation}  \label{tower_ineq_0 s-tGeneralVersion}
\overline{\sigma}_{\rho}( \overline{\sigma}_{\tau}~f)(S) = \overline{\sigma}_{\rho}[{\bf 1}_{\{\tau = \infty\}}~f+{\bf 1}_{\{\tau < \infty\}}~\overline{\sigma}_{\tau} f](S)\le \overline{\sigma}_{\rho} f(S).
\end{equation}
and, consequently,
\begin{equation} \nonumber  
 \underline{\sigma}_{\rho} f(S) \leq \underline{\sigma}_{\rho}[{\bf 1}_{\{\tau = \infty\}}~f+{\bf 1}_{\{\tau < \infty\}}~\underline{\sigma}_{\tau} f](S) = \underline{\sigma}_{\rho}( \underline{\sigma}_{\tau}~f)(S).
\end{equation}
\end{theorem}
\begin{proof}
To establish (\ref{tower_ineq_0 s-tGeneralVersion}), we argue as follows:
\begin{equation} \nonumber
\overline{\sigma}_{\tau}  f= {\bf 1}_{\{\tau= \infty\}}~\overline{\sigma}_{\tau}(f)
+ {\bf 1}_{\{\tau <\infty\}}~\overline{\sigma}_{\tau}(f) = 
{\bf 1}_{\{\tau= \infty\}}~\overline{\sigma}_{\infty} f
+ {\bf 1}_{\{\tau <\infty\}}~ \overline{\sigma}_{\tau} f
= {\bf 1}_{\{\tau= \infty\}}~f
+ {\bf 1}_{\{\tau <\infty\}}~ \overline{\sigma}_{\tau} f,
\end{equation}
The first equality in (\ref{tower_ineq_0 s-tGeneralVersion}) is then obvious.  We now concentrate in establishing the second inequality in (\ref{tower_ineq_0 s-tGeneralVersion}).

Consider first the case of $\rho(S) = \infty$ then  the right hand side of (\ref{tower_ineq_0 s-tGeneralVersion}) equals $f(S)$ while:
${\bf 1}_{\{\tau = \infty\}}(S)~f(S)+{\bf 1}_{\{\tau < \infty\}}(S)~\overline{\sigma}_{\tau} f(S)= f(S)$ where we used $ \infty =\rho(S) \leq \tau(S)$.

Consider now $\rho(S) < \infty$; we may assume without loss of generality that $\overline{\sigma}_{\rho}f(S) < \infty$.
Since $\overline{\sigma}_{\rho} g(S)= \overline{\sigma}_{\rho(S)} g(S)$ for any $g \in Q$ it should not lead to confusion to set $j=\rho(S)$, we will then be conditioning on the node $(S,j)= (S, \rho(S))$.

We will use the notation $f_m=\liminf\limits_{n\to \infty}~~\Pi^{V^m, H^m}_{j,n}$, $m \geq 0$, with $H_i^0\equiv 0$ for $i\ge n_0$,  and assume
\begin{equation}  \label{firstNodewiseBound}
f\le \sum\limits_{m\ge 0} f_m~~\mbox{on}~~\Se_{(S,j)},~
\Pi^{V^0, H^0}_{j,n}\in \mathcal{E}_{(S,j)}\;\; \;~\mbox{and for}~~ m\ge 1,\;\; \Pi^{V^m, H^m}_{j,n}\in \mathcal{E}^+_{(S,j)},
\end{equation}
for any  $n\geq j$. We note that  that $(H^m_i|_{\Se_{(S,k)}})_{i\ge k}\in\He_{(S,k)}$for any $S\in\Se$ (see Definition 2.3 in \cite{bender3}). We let
\begin{equation}  \nonumber 
\tilde{\mathcal{S}}_{(S,j)} = \mathcal{S}_{(S,j)} \setminus \{\tilde{S} \in \mathcal{S}_{(S,j)}:~~\tau(\tilde{S}) = \infty \}
\end{equation}
and  define the functions $U^m$ on $\mathcal{S}_{(S,j)}$, by
\begin{equation} \nonumber
U^m(\tilde{S})\equiv \Pi^{V^m, H^m}_{j,\tau(\tilde{S})}(\tilde{S}),\;\;\tilde{S}\in\tilde{\mathcal{S}}_{(S,j)}, ~\mbox{and}~~U^m(\tilde{S})= f_m~\mbox{when}~~\tau(\tilde{S}) = \infty,
\end{equation}
these functions are constant on $\Se_{(\tilde{S},\tau(\tilde{S}))}$ whenever $\tau(\tilde{S}) < \infty$: $S_i'=\tilde{S}_i,\; 0\le i\le \tau(\tilde{S})\Rightarrow \tau(S')=\tau(\tilde{S})$. Also note that $\tau(\tilde{S}) \geq \rho(\tilde{S}) = \rho(S)=j$ where we used the stopping time property of $\rho$.  Moreover $0\le i < \tau(\tilde{S})\Rightarrow H^m_i(S')=H^m_i(\tilde{S})\;\&\; \Delta_iS'=\Delta_i\tilde{S}$. From our definition in (\ref{atTheEndOfTime}), $\mathcal{S}_{(\tilde{S}, \infty)} \equiv \{\tilde{S}\}$  so $U^m$ is also constant in that space which corresponds to
$\tau(\tilde{S}) = \infty$. It follows that, for each $\tilde{S}\in \tilde{\Se}_{(S,j)}$,  the functions $f_m=\liminf\limits_{n\to \infty}~~\Pi^{U^m(\tilde{S}), H^m}_{\tau(\tilde{S}),n}$, for  $m \geq 0$,  are defined on $\Se_{(\tilde{S},\tau(\tilde{S}))}$ and satisfy,
$\Pi^{U^0(\tilde{S}), H^0}_{\tau(\tilde{S}),n}\in \mathcal{E}_{(\tilde{S},\tau(\tilde{S}))}$ and for $m\ge 1,\; \Pi^{U^m(\tilde{S}), H^m}_{\tau(\tilde{S}),n}\in \mathcal{E}^+_{(\tilde{S},\tau(\tilde{S}))}$. To see this we compute for $S'\in\Se_{(\tilde{S},\tau(\tilde{S}))}$, $\Pi^{U^m(\tilde{S}), H^m}_{\tau(\tilde{S}),n}(S')=U^m(\tilde{S})+ \sum\limits_{i=\tau(\tilde{S})}^{n-1}H^m_i(S')\Delta_iS' =V^m+\sum\limits_{i=j}^{\tau(\tilde{S})-1}H^m_i(S')\Delta_iS'+\sum\limits_{i=\tau(\tilde{S})}^{n-1}H^m_i(S')\Delta_iS'=\Pi^{V^m, H^m}_{j,n}(S')$  and then we rely on (\ref{firstNodewiseBound}). We have used $j=\rho(S)=\rho(\tilde{S})\le \tau(\tilde{S})$.

Using (\ref{firstNodewiseBound}) and the above definition of $U^m$ we have
\begin{equation} \nonumber 
f \leq \sum_{m \geq 0} \liminf_{n \rightarrow \infty}~ \Pi^{U^m, H^m}_{\tau(\tilde{S}),n}~~\mbox{on}~\mathcal{S}_{(\tilde{S}, \tau(\tilde{S}))}~~\mbox{for all}~\tilde{S}~\in \tilde{\mathcal{S}}_{(S,j)},
\end{equation}
therefore
\begin{equation} \label{mainInequality}
\overline{\sigma}_{\tau(\tilde{S})} f (\tilde{S}) \le \sum_{m\ge 0} U^m(\tilde{S}),~~\mbox{for all}~\tilde{S}~\in \tilde{\mathcal{S}}_{(S,j)}.
\end{equation}
\indent On $\mathcal{S}_{(S,j)}$ define the portfolios
\[
G^m_i(\tilde{S})=\left\{
\begin{array}{lll}
H^m_i(\tilde{S}) & \mbox{if} & j\le i < \tau(\tilde{S})\\
0 & \mbox{if} & \tau(\tilde{S})\le i
\end{array}
\right. \quad m\ge 0,
\]
which are non-anticipative on $\mathcal{S}_{(S,j)}$  by Lemma \ref{st-portfolio} with $H^1=H^m, H^2=0$, and $H^{\tau}_i=G^m_i$.

It then follows that $U^m(\tilde{S})=\Pi^{V^m, H^m}_{j,\tau(\tilde{S})}(\tilde{S})=\liminf\limits_{n\to \infty}~~\Pi^{V^m, G^m}_{j,n}(\tilde{S})$ whenever $\tilde{S} \in \tilde{\mathcal{S}}_{(S,j)}$ and $m \geq 0$. Moreover, from our definition of $U^m$ we also have
$U^m(\tilde{S})=\liminf\limits_{n\to \infty}~~\Pi^{V^m, G^m}_{j,n}(\tilde{S})$, $m \geq 0$, whenever $\tau(\tilde{S})= \infty$.
Also, from (\ref{firstNodewiseBound}) we get
\[
\Pi^{V^0, G^0}_{j,n}\in \mathcal{E}_{(S,j)},\;\;\mbox{and}\;\;\; \Pi^{V^m, G^m}_{j,n}\in \mathcal{E}^+_{(S,j)}, m \geq 1.
\]
Therefore
\begin{equation}  \label{hereWeUseTheCharacteristicFunction}
{\bf 1}_{\{\tau = \infty\}}(\tilde{S})~ ~f(\tilde{S})+ {\bf 1}_{\{\tau < \infty\}}(\tilde{S})~ \overline{\sigma}_{\tau(\tilde{S})} f(\tilde{S}) \leq
\sum_{m \geq 0} U^m(\tilde{S}) = \sum_{m \geq 0} \liminf\limits_{n\to \infty}~~\Pi^{V^m, G^m}_{j,n}(\tilde{S})~~\mbox{for all}~\tilde{S}~\in~\mathcal{S}_{(S,j)};
\end{equation}
to check this inequality notice that whenever ${\bf 1}_{\{\tau < \infty\}}(\tilde{S})=1$ it follows from (\ref{mainInequality}) and whenever ${\bf 1}_{\{\tau = \infty\}}(\tilde{S})=1$ we note that
the right hand side of (\ref{hereWeUseTheCharacteristicFunction}) equals $\sum_{m \geq 0} f_m$ which from (\ref{firstNodewiseBound}) is greater or equal than $f$.
From (\ref{hereWeUseTheCharacteristicFunction}), and previous remarks,
it follows that
 $\overline{\sigma}_j[{\bf 1}_{\{\tau = \infty\}}~ ~f + {\bf 1}_{\{\tau < \infty\}}~\overline{\sigma}_{\tau}f](S) \le \sum\limits_{m\ge 0} V^m(S)$ and so
\[  \label{tower_ineq_1}
\overline{\sigma}_{\rho}[{\bf 1}_{\{\tau = \infty\}}~ ~f + {\bf 1}_{\{\tau < \infty\}}~ \overline{\sigma}_{\tau} f](S) \le\overline{\sigma}_{\rho} f(S).
\]

\end{proof}

The proof of the next corollary follows the same steps as the proof of Theorem~\ref{mainDirectionOfTowerProperty s-t}. 
\begin{corollary}[{\bf Tower properties for $\overline I$}]
\label{towerPropertyForI-bar}
A version of Theorem~\ref{mainDirectionOfTowerProperty s-t}  holds under the same assumptions if one replaces $\overline{\sigma}$ by $\overline{I}$ and assumes in addition that $f \in P$.
\end{corollary}

\begin{corollary}
\label{a.eAndStoppingTimes}
Assume, $f, g \in Q$, $f= g$ $a.e.$ and that $\rho$ is a stopping time. Then
\begin{equation} \nonumber
\overline{\sigma}_{\rho} f = \overline{\sigma}_{\rho} g ~~~a.e.  
\end{equation}
\end{corollary}
\begin{proof}
From Corollary \ref{towerPropertyForI-bar}: $0 \leq \overline{I}(\overline{I}_{\rho}(|f-g|) \leq\overline{I}(|f-g|)=0$; therefore
\begin{equation} \nonumber 
\overline{I}_{\rho}(|f-g|)=0 ~a.e.
\end{equation} 
The above equation means that there exists $N \subseteq \mathcal{S}$ such that $\overline{I}({\bf 1}_N)=0$ and $\overline{I}_{\rho(S)}(|f-g|)(S)= 0 $ for all $ S \in N^C \equiv \mathcal{S} \setminus N$. Therefore for all $S\in N^C$ we have
\begin{equation} \nonumber
 f =g~\mbox{holds}~ \overline{I}_{(S, \rho(S))}-a.e. ~\mbox{on}~\mathcal{S}_{(S, \rho(S))}~(\mbox{i.e., the equality holds a.e. in } \mathcal{S}_{(S, \rho(S))}).
 \end{equation}
 It then follows from Proposition A.2.1, item $b)$, from \cite{bender3} that $\overline{\sigma}_{\rho(S)}f (S)=
 \overline{\sigma}_{\rho(S)}g (S)$ for all $S \in  N^C$. 
\end{proof}

\section{Doob's Optional Stopping} \label{optionalStopping}
Here we go through a series of results leading
to our versions of  Doob's optional sampling theorem. We will organize the results in an analogous way to how classical (i.e. probability-theory based) results are handled. First we refer to some general facts.

\begin{definition}[Stopped Sequence] \label{stoppedSequence}
Given a sequence $\{f_j\}_{j \geq 0},~f_j \in Q$  and a stopping time $\tau$, we define the stopped sequence $\{f_{ j}^{\tau}\}_{j \geq 0 }$ by
\begin{equation} \nonumber
f_j^{\tau}(S):=f_{\tau(S)\wedge j}(S).
\end{equation}
Furthermore, if $\tau< \infty$ a.e. then  $f_{\tau}(S):=
f_{\tau(S)}(S)$ is defined a.e. 
\end{definition}
    
\begin{remark} \label{easyStoppingResults}
Given a supermartingale sequence $\{f_j\}$ the stopped sequence $\{f^{\tau}_j\}$ is a supermartingale (as per \cite[Example 3, item d)]{bender3}). As supermartingales and submartingales are related by the duality: $\overline{\sigma}_j ~\rightarrow \underline{\sigma}_j$, $f \rightarrow -f$ it follows that a stopped submartingale is also a submartingale. It then follows that 
if $(L)-a.e.$  is assumed (which by \cite[Proposition 3.3]{bender3} will imply $\underline{\sigma}_jf \leq \overline{\sigma}_jf~~a.e.$
for all $f \in Q$) we have that a stopped martingale is a martingale as well.
\end{remark}

\subsection{Supermartingales of the form $f_j = \overline{\sigma}_j f$}

For a fixed $f \in Q$, the sequence $f_j = \overline{\sigma}_j f$,$ j \geq 0$ is a supermartingale sequence with $\overline{\sigma}_j f_{j+1} \leq f_j$ holding everywhere on $\mathcal{S}$ (as per \cite[Example 3, item b)]{bender3}).  \cite[Corollary 5.6]{bender3} gives general conditions for $\{\overline{\sigma}_j f\}$ being a martingale sequence.
%

This subsection involves the special case of a supermartingale  $f_j\equiv \overline{\sigma}_j f$, later versions of Doob's optional sampling  generalize the result to more general supermartingale sequences (but relying on stronger hypothesis). This restriction allows us to impose no conditions on $\mathcal{S}$ nor do we need $f$ to be non-negative.




\begin{theorem}[Doob's optional stopping version I, supermartingale $\{f_j \equiv \overline{\sigma}_j f\}$] \label{doobIIWithTwoStoppingTimes} \hspace{.1in}

\noindent
Let $f \in Q$, $\rho \leq \tau$ be two stopping times and set $f_j = \overline{\sigma}_j f$, $j \geq 0$. Then
\begin{enumerate}
\item
\begin{equation} \nonumber
\overline{\sigma}_{\rho} f^{\tau}_k \leq f^{\rho}_k\;\;\mbox{for all}\;\; k \geq 0.
\end{equation}

\item 
\begin{equation} \nonumber
\overline{\sigma}_{\rho} f_{\tau} \leq f_{\rho}~\mbox{on}~\mathcal{S}.
\end{equation}
\end{enumerate}
\end{theorem}

\begin{proof}
Since $f^{\tau}_k=f_{\tau\wedge k}=\overline{\sigma}_{\tau\wedge k}f$, in case that $\rho(S) \leq k$,  the first statement follows from (\ref{tower_ineq_0 s-tGeneralVersion}) in Theorem \ref{mainDirectionOfTowerProperty s-t}, with $\tau\wedge k$ in place of $\tau$.
In case that
$ k < \rho(S)$ we note that if $\hat{S} \in \mathcal{S}_{(S, \rho(S))}$ we have
$ k < \rho(S)= \rho(\hat{S}) \leq \tau(\hat{S})$ and so $f^{\tau}_{k}(\hat{S})
= f_{k}(\hat{S})$. Therefore $\overline{\sigma}_{\rho}f^{\tau}_{k}(S) = \overline{\sigma}_{\rho}(f_{k})(S) \leq f_k(S) = f^{\rho}_k(S)$ where the inequality holds as
$f_k$ is constant on $\mathcal{S}_{(S, \rho(S))}$ .

\noindent
For the second statement, we apply \eqref{tower_ineq_0 s-tGeneralVersion} to conclude that
$$
\overline{\sigma}_\rho f_\tau= \overline{\sigma}_\rho \overline{\sigma}_\tau f \leq \overline{\sigma}_\rho  f=f_\rho
$$
valid on $\Se$.

\end{proof}

\subsection{General Supermartingale $\{f_j\}$ and bounded stopping times $\rho \leq \tau$}

The next result generalizes, under appropriate conditions, the previous theorem from supermartingales of the form
$f_j = \overline{\sigma}_j f$ to general supermartingale sequences $f_j$.
Simple trajectory based conditions implying that $(L)-a.e.$ holds (a property required in the following theorem) are presented in \cite{bender3}[corollaries 3.13 and  3.14].

\begin{theorem}[Doob's optional stopping version II, general supermartingale] \label{doobIIIWithTwoStoppingTimes general}
Let $\{f_j\}_{j \geq 0}$ be a given  real-valued supermartingale on $\mathcal{S}$. Consider two stopping times satisfying  
$\rho\le \tau$ a.e. and assume that $(L)-a.e$ holds. Then, if $\tau \leq C ~~a.e.$, where $C$ is a non negative integer:
\begin{equation}\label{mainStatement}
\overline{\sigma}_{\rho} f_{\tau} \leq f_{\rho}\;\;\mbox{a.e. on}~~\mathcal{S}.
\end{equation}
\end{theorem}
\begin{proof}
We could assume for the purpose of the proof that $\rho \leq \tau$ on all of $\mathcal{S}$; otherwise, replace $\tau$ by $\tau'=\rho\vee\tau$. Then $\tau'$ is a stopping time,
$\rho\le \tau'$ on $\mathcal S$, and $\tau'=\tau$ a.e. Hence
$f_{\tau'}=f_\tau$ a.e., and Corollary~\ref{a.eAndStoppingTimes} gives
\[
\overline{\sigma}_\rho f_{\tau'}
=
\overline{\sigma}_\rho f_\tau
\quad a.e.
\]
Thus it is enough to prove the results under the stronger assumption
$\rho\le\tau$ everywhere.

\noindent
{\bf Case $\tau \leq C$ on $\mathcal{S}$}. We rely on Theorem \ref {thm:decomposition}, writing $\tau= \tau(S)$ and $\rho = \rho(S)$;
from (\ref{supermartingaleDecomposition}) it follows that
the following equality is valid for all $S \in N^C_f$ (where $N_f$ is the null set appearing in (\ref{supermartingaleDecomposition})),
\begin{equation} \nonumber 
f_{\tau}(S)= f_{\rho}(S) + \sum_{j= \rho}^{\tau -1} H_j(S) \Delta_j S + \sum_{j= \rho}^{\tau -1} \delta_j- A_{\tau}(S) + A_{\rho}(S).
\end{equation}
Therefore, the following inequality holds for all $\hat{S} \in \mathcal{S}_{(S, \rho(S))}$ and for any $S \in \mathcal{S}$
\begin{equation} \label{superhedgingTheSupermartingale}
f_{\tau}(\hat{S})\leq  f_{\rho}(S) + \sum_{j= \rho}^{C} H_j(\hat{S}) \Delta_j \hat{S} + \delta+ \infty~{\bf 1}_{N_f}(\hat{S}),
\end{equation}
where $\delta \equiv \sum_{j \geq 0} \delta_j$ and we extend the strategy by $H_j \equiv 0$, $j \geq \tau$,  so that $\sum_{j= \rho}^{\tau -1} H_j(\tilde{S}) \Delta_j S = \sum_{j= \rho}^{C} H_j(\tilde{S}) \Delta_j S$. By Lemma \ref{st-portfolio}, the resulting zero-padded strategy is still a well-defined portfolio with deterministic maturity $C$. Applying $\overline{\sigma}_{\rho(S)}$ to (\ref{superhedgingTheSupermartingale}) we obtain
\begin{equation} \nonumber 
\overline{\sigma}_{\rho} f_{\tau}(S)\leq  f_{\rho}(S) + \delta+ \overline{I}_{\rho} (\infty~{\bf 1}_{N_f})(S).
\end{equation}
We then compute, $0 \leq \overline{I}(\overline{I}_{\rho} (\infty~{\bf 1}_{N_f})) \leq \overline{I}(\infty~{\bf 1}_{N_f})=0$ where the second inequality follows from Corollary \ref{towerPropertyForI-bar} applied to the stopping times $0$ and $\rho$ (notice that $0 \leq \rho$) and the equality follows from the fact  that  $g \equiv \infty~{\bf 1}_{N_f}$ is a null function. Therefore, $\overline{I}_{\rho} (\infty~{\bf 1}_{N_f})$ is a null function and hence equals the $0$ function a.e. It then follows that $\overline{\sigma}_{\rho} f_{\tau} \leq  f_{\rho} + \delta$
holds a.e.,  from this fact
we then conclude the proof  of (\ref{mainStatement}) by letting $\delta\downarrow 0$.

\vspace{.1in}
\noindent
{\bf General Case $\tau \leq C~a.e.$ } 
Note that $f_\tau=f_{\tau\wedge C}$ a.e. Hence, by Corollary \ref{a.eAndStoppingTimes}, 
\begin{equation*}
	\overline{\sigma}_{\rho \wedge C}(f_{\tau})= \overline{\sigma}_{\rho \wedge C}(f_{\tau\wedge C})\leq f_{\rho  \wedge C},\quad \textnormal{a.e.}
\end{equation*}
applying the previous step to the bounded stopping times $\tau \wedge C$ and $\rho \wedge C$. It then follows that $\overline{\sigma}_{\rho}(f_{\tau}) \leq f_{\rho}$ a.e.,
 recalling that $\rho \leq C$ a.e.

\end{proof}

\subsection{Positive Supermartingale $\{f_j\}$ and general stopping times $\rho \leq \tau $}

Finally we present the case of a general non-negative supermartingale (real valued), the stopping times are left general but we need to assume the aditional property $(P)$ (see Definition \ref{def:Positivity}) to deal with unbounded time.

\begin{theorem}[Doob's optional stopping version III] \label{towerPropertyForgeneralSupermartingales}
Let $\{f_j\}_{j \geq 0}$ be a given non-negative,  real-valued supermartingale on $\mathcal{S}$ we also set $f_{T+1} \equiv \liminf_{n \rightarrow T+1} f_n$. Consider two stopping times satisfying
$\rho\le \tau$ a.e. Under the assumption that both,  $(L)-a.e$ and property $(P)$ (see Definition \ref{def:Positivity}), then:
\begin{enumerate}
\item 
\begin{equation} \label{generalizedTowerInequality3}
\overline{\sigma}_{\rho}(f_{\tau}) \leq f_{\rho}~~\mbox{a.e. on}~~\mathcal{S},
\end{equation}
\item 
\begin{equation} \label{generalizedTowerInequality4}
\overline{\sigma}_{\rho}({\bf 1}_{\{\tau < T+1\}}~f_{\tau}) \leq {\bf 1}_{\{\rho < T+1\}} f_{\rho}~\mbox{~a.e. on}~~\mathcal{S}.
\end{equation}
\end{enumerate}
\end{theorem}
\begin{proof}
As in the proof of Theorem~\ref{doobIIIWithTwoStoppingTimes general}, replacing
$\tau$ by $\rho\vee\tau$ reduces the case $\rho\le\tau$ a.e. to the case
$\rho\le\tau$ everywhere. We therefore assume $\rho\le\tau$ on all of $\mathcal{S}$ throughout the proof.
If $\rho(S) = \infty$ we recall that $\mathcal{S}_{(S, \infty)} \equiv \{S\}$
and $\overline{\sigma}_{\infty}$ is the identity operator. Hence,
 the left hand side of (\ref{generalizedTowerInequality4})
involves ${\bf 1}_{\{\tau < \infty\}}(S)~f_{\tau}(S)= 0~f_{\infty}(S)$ which
makes the first inequality in (\ref{generalizedTowerInequality4}) into an actual equality.
Similarly (\ref{generalizedTowerInequality3}) becomes $f_{\infty} = f_{\infty}$.
 
\vspace{.1in}
Consider then $\rho(S) < \infty$, we rely on Theorem \ref {thm:decomposition}, writing $\tau= \tau(S)$ and $\rho = \rho(S)$;
from (\ref{supermartingaleDecomposition}) it follows that the following equality is valid for all $\tilde{S} \in (N^C_f \cap \mathcal{S}_{(S, \rho)})$ (where $N_f$ is the null set appearing in (\ref{supermartingaleDecomposition})).
\begin{equation}   \nonumber
f_{\tau \wedge n}(\tilde{S}) \leq  f_{\rho}(S) + \sum_{j= \rho}^{\tau -1 \wedge n-1} H_j(\tilde{S}) \Delta_j \tilde{S} + \delta, \quad n\geq \rho(S),
\end{equation}
where we have used the fact that the sequence of functions $A_j$ is non-decreasing and set $\delta \equiv \sum_{j \geq 0} \delta_j$. 
Then
\begin{equation} \nonumber 
f_{\tau}(\tilde{S}) \leq f_{\rho}(S) + \liminf_{n \rightarrow \infty}\sum_{j= \rho}^{\tau -1 \wedge n-1} H_j(\tilde{S}) \Delta_j \tilde{S} + \delta.
\end{equation}
Therefore, the following inequality holds for all $\hat{S} \in \mathcal{S}_{(S, \rho(S))}$ and for any $S \in \mathcal{S}$ satisfying $\rho(S) < \infty$
\begin{eqnarray} \label{oneNeedsToAddArhoButOK}\nonumber
f_{\tau}(\hat{S})&\leq & f_{\rho}(S) + \delta+ \liminf_{n \rightarrow \infty} \sum_{j= \rho}^{\tau -1 \wedge n-1} H_j(\hat{S}) \Delta_j \hat{S} + \infty~{\bf 1}_{N_f}(\hat{S}) \\&=& 
f_{0}(S) + \liminf_{n \rightarrow \infty} \sum_{j= 0}^{n-1} H_j(\hat{S}) \Delta_j \hat{S} + \delta+ f_{\rho}(S)- f_0(S) - \sum_{j=0}^{\rho -1} H_j(\hat{S}) \Delta_j \hat{S}+ \infty~{\bf 1}_{N_f}(\hat{S}).
\end{eqnarray}

From  \cite[Lemma 7.3]{bender3}, $f_{0}(S)+ \delta+ \sum_{j= 0}^{n-1} H_k(S)  \Delta_k S
\geq 0$ holds for all $S \in \mathcal{S}$ and for all $n \geq 0$. In particular, $f_{0}(S)+ \delta+ \sum_{j= 0}^{\rho-1} H_k(S)  \Delta_k S +\sum_{j= \rho}^{n-1} H_k(\hat{S})  \Delta_k \hat{S}
\geq 0$ for all $\hat{S} \in \mathcal{S}_{(S, \rho)}$ and for all $n \geq \rho(S)$. 
Therefore applying
 $\overline{\sigma}_{\rho}$ to (\ref{oneNeedsToAddArhoButOK})
\begin{eqnarray*} \nonumber 
&& \overline{\sigma}_{\rho} f_{\tau}(S) \leq f_{0}(S)+ \delta+ \sum_{j= 0}^{\rho-1} H_j(S)  \Delta_j S +
\overline{\sigma}_{\rho}(f_{\rho}(S)- f_0(S) - \sum_{j=0}^{\rho -1} H_j(S) \Delta_j S) (S)+ \overline{I}_{\rho} (\infty~{\bf 1}_{N_f})(S) \\ &\leq &
f_{\rho}(S) + \delta +\overline{I}_{\rho} (\infty~{\bf 1}_{N_f})(S).
\end{eqnarray*} 
It then follows that
\begin{equation} \label{previousToConclusion}
\overline{\sigma}_{\rho}( ~f_{\tau})\leq  f_{\rho} + \delta, ~~\mbox{holds on}~~\mathcal{M}^C \cap \{S \in \mathcal{S}: \rho(S) < \infty\}
\end{equation}
where $\mathcal{M}$ is a null set such that $\overline{I}_{\rho} (\infty~{\bf 1}_{N_f})(S)=0$
for $S\in \mathcal{M}^C$. This function is indeed a null function, because $0\leq \bar I(\bar{I}_{\rho} (\infty~{\bf 1}_{N_f})) \leq  \overline I( \infty~{\bf 1}_{N_f})=0$ by Corollary \ref{towerPropertyForI-bar}. But, notice that (\ref{previousToConclusion}) also holds in the set $\{\rho = \infty\}$ and so it holds in all of $\mathcal{M}^C$, one can then conclude 
the proof  of (\ref{generalizedTowerInequality3}).
 Finally, (\ref{generalizedTowerInequality4}) holds for $S$ such that $\rho(S) = \infty$, it then remains to check it for the case $\rho(S) < \infty$. As  $f_n \geq 0$, we have
 ${\bf 1}_{\{\tau < \infty\}} f_{\tau} \leq {\bf 1}_{\{\rho < \infty\}} f_{\tau}$ and so
 by means of Lemma \ref{takingOutCharacteristicFunctions}
 and (\ref{generalizedTowerInequality3}) we compute
 \begin{equation} \nonumber
\overline{\sigma}_{\rho(S)} ({\bf 1}_{\{\tau < \infty\}} f_{\tau})(S) \leq \overline{\sigma}_{\rho(S)} ({\bf 1}_{\{\rho < \infty\}} f_{\tau})(S) \leq {\bf 1}_{\{\rho < \infty\}}(S) \overline{\sigma}_{\rho(S)} ( f_{\tau})(S) \leq {\bf 1}_{\{\rho < \infty\}}(S)  f_{\rho}(S)
 \end{equation}
 where the last inequality only holds on $\mathcal{M}^C \cap \{\rho < \infty\}$ where $\mathcal{M}$ is the null set where (\ref{generalizedTowerInequality3}) holds. Having already argued that (\ref{generalizedTowerInequality4}) holds on $\{\rho = \infty\}$ is follows from the above inequalities that holds on all of $\mathcal{M}^C $.
\end{proof}

\section{Dubins' Inequality} \label{dubinsInequality}

Here we prove a classical inequality of Dubins in our non-probabilistic/trajectorial setting. Dubins' classical result carries some empirical content as it expresses a sharp
bound for the superhedging price for the characteristic function of the set of upcrossings. 
Theorem \ref{dubinFiniteTime} below requires minimal hypothesis on the trajectory set hence, one expects, the result to be valid for a potential sequence of market prices. It is then possible to speculate that it will be {\it unlikely} for those observable trajectories to upcross a large number of times a given interval. The meaning of unlikely is here a natural extension of referring to arbitrage nodes (see Definition \ref{typesOfNodes}) as being unlikely because of the fact that they are null sets.
We first handle
the finite time case which is of main interest given the meaning of the result and also requires weaker hypothesis than the $T=\infty$ case.


For a  sequence of real valued functions $\{f_i\}_{0 \leq i < T+1}$, $T$ an integer constant with
$T =\infty$ a possibility, defined on $\Se$  and a real interval $[a,b],\;0 \leq a < b$, \emph{the band where upcrossings are counted}.
We will allow for two possible cases $T < \infty$ and $T= \infty$ in the setting
and will explicitly indicate when a particular result may assume $T < \infty$
(referred as {\it finite time}). We will be counting upcrossings for $f_0(S), \ldots, f_T(S)$
in the finite case $T < \infty$ but the stopping times may take the value $T+1$ and that is why we need non-essential access to $f_{T+1}$. 

\noindent
In what follows, let $\{f_i\}_{0 \leq i < T+1}$ be a sequence of non-anticipative functions, i.e. $f_i(S)=f_i(S_0,...,S_i)$. In particular, defined objects and results will depend on $\{f_i\}_{0 \leq i < T+1}$, as well as on $T$, although this may not be reflected in the notation.

\vspace{.05in}
Next we define the upcrossing times
$\rho_{k}, \tau_{k}:\Se \to \mathbb{N} \cup \{\infty\},\; k\ge 0$. 
\begin{definition}[Upcrossing Times]\label{upcrossing_stopping_times}
Set $\tau_0\equiv 0$. For any $S\in\Se$ define recursively,
\begin{equation} \nonumber
\rho_k(S) \equiv (\inf\{n: \tau_k(S)\le n,~ f_n(S) \leq a\}) \wedge (T+1),\; \mbox{for}\; k\ge 0.
\end{equation}
\begin{equation}  \nonumber
\tau_k(S) \equiv (\inf\{n: \rho_{k-1}(S)\le n,~ f_n(S) \geq b\}) \wedge (T+1),\; \mbox{for}\; k\ge 1.
\end{equation}
\end{definition}
We use the convention $\inf\emptyset=\infty$; hence, after the truncation by $T+1$, the corresponding crossing time equals $T+1$ whenever the relevant crossing set is empty. Notice that  $\rho_{k-1} \leq \tau_k \leq \rho_k$, $k \geq 1$.
The functions $\rho_k, \tau_k$ are stopping times, see Lemma 7 in \cite{ferrando} which covers the case $T= \infty$ but the finite time case then follows from the Definition \ref{upcrossing_stopping_times} and fact that the minimum of two stopping times is a stopping time.

We will say that $\rho_k$ is the time of the k-th $a$-drop and
$\tau_k$ is the time of the k-th $b$-increase, for $k \geq 1$. Notice that
${\bf 1}_{\{\tau_{k} < T+1\}}(S)=1$ if and only if  $\{f_j(S)\}_{0 \leq i < T+1}$
upcrosses the interval $[a,b]$ at least $k$ times.

In the sequel, for  simplicity,  once $S$ is clearly understood, we may write in some occasions $\rho_k, \tau_{k}$ instead of $\rho_k(S),\tau_{k}(S)$ for any $k\ge0$.

Most definitions in this section assume as given, sometimes implicitly, a sequence of functions $\{f_j\}_{0 \leq j < T+1}.$

\begin{definition}[Upcrossings Counting Function]\label{upcrossingDefinition}
For $0 \leq n < T+1$ and $S \in \mathcal{S}$, denote by $U^{[a,b]}_n(S)=U_n(S)$ the number of upcrossings of the sequence $\{f_i(S)\}_{i=0}^n$ through the interval $[a,b]$ and so given by:
\[U_n(S) =\max\{k\in \mathbb{N}: \tau_{k}(S)\le n\}.\]
The total number of upcrossings for $\{f_i(S)\}_{0 \leq i < T+1}$ is defined by
\[U^{[a,b]}(S)=U(S)= \sup\{U_n(S): 0 \leq n < T+1\}.\]
\end{definition}
\noindent Notice that if $0 \le n\le m$ then $0 \leq U_n\le U_m$ on $ \Se$.

\begin{theorem}[Dubins' inequality for finite time, $T < \infty$] \label{dubinFiniteTime}
Let $f_0, f_1, \ldots, f_T,$ be a given real valued supermartingale sequence with $f_n \geq c $ where $c$ is a constant. Assume the definitions of $\rho_k, \tau_k$ given by Definition \ref{upcrossing_stopping_times}  with $T < \infty$ (i.e. the finite time case). Furthermore, assume
that $(L)-a.e.$ holds; then,
for any $k\ge 0$ and $c \leq a < b$:
\begin{equation} \nonumber 
\overline{\sigma}({\bf 1}_{\{\tau_{k+1} < T+1\}}) \leq \left(\frac{a-c}{b-c} \right)~~ \overline{\sigma}({\bf 1}_{\{\rho_{k}  < T+1\}}) \leq
\left(\frac{a-c}{b-c}\right)~~ \overline{\sigma}({\bf 1}_{\{\tau_{k}  < T+1\}}),
\end{equation}
and so
\begin{equation} \label{couldBeCheckedEmpirically}
\overline{\sigma}({\bf 1}_{\{\tau_{k+1}  < T+1\}}) \leq
\left(\frac{a-c}{b-c} \right)^{k+1} ~\overline{\sigma}({\bf 1}_{\{\rho_{0}  < T+1\}}) \leq
\left(\frac{a-c}{b-c}\right)^{k+1}.
\end{equation}
\end{theorem}

\begin{proof}
It is  enough to provide a proof for the case $c=0$ as the general case
can be obtained from it.
For the purpose of the proof we set $f_{T+1} \equiv f_T$.

We will prove: $\overline{\sigma}({\bf 1}_{\{\tau_{k+1} < T+1\}}) \leq
\frac{a}{b}~~ \overline{\sigma}({\bf 1}_{\{\rho_{k} < T+1\}})$ and remark that
$\{\rho_{k} < T+1\} \subseteq
\{\tau_{k} < T+1\}$. These facts will complete the proof of the theorem.

\vspace{.1in}
From  Theorem \ref{thm:decomposition} we obtain the following equality being valid in the complement of a null set $\mathcal{M}$
\begin{equation} \nonumber
    f_{\tau_{k+1}}(S)=  f_{\rho_{k}}(S)+ \sum_{i= \rho_k}^{\tau_{k+1}-1} H_i(S) \Delta_i S - A_{\tau_{k+1}}(S)+ A_{\rho_{k}}(S)+ \sum_{i= \rho_k}^{\tau_{k+1}-1} \delta_i.
\end{equation}
Therefore, given that $f_n \geq 0$, the following inequality holds on $\mathcal{S}$
\begin{equation} \label{firstStep1}
{\bf 1}_{\{\tau_{k+1} < T+1\}}(S)~~f_{\tau_{k+1}}(S) \leq
{\bf 1}_{\{\rho_{k} < T+1\}}(S)~f_{\rho_{k}}(S)~+~ \sum_{i=0}^T \hat{H}_i(S) ~\Delta_i S+ \delta+ \infty ~{\bf 1}_{\mathcal{M}}(S)
\end{equation}
where $\hat{H}_i \equiv {\bf 1}_{\{\rho_{k} < T+1\}}~H_i$ whenever $\rho_k \leq i \leq \tau_{k+1}-1$ and  $\hat{H}_i= 0$ otherwise.

We now apply $\overline{\sigma}$ to (\ref{firstStep1}) and use subadditivity
of this functional and the fact that $\overline{\sigma }(\infty {\bf 1}_{\mathcal{M}}) \leq
\overline{I}(\infty {\bf 1}_{\mathcal{M}})=0$
 to get
\begin{equation} \nonumber 
\overline{\sigma}({\bf 1}_{\{\tau_{k+1} < T+1\}})\leq
~\frac{a}{b} \overline{\sigma}({\bf 1}_{\{\rho_{k} < T+1\}}) + \frac{\delta}{b},
\end{equation}
which concludes the proof as $\delta$ is arbitrary.
\end{proof}


\vspace{.1in}
Corollary \ref{dubinMovingBandsCorollary} below
will make use of the following definition which counts upcrosses but, this time,
through bands that evolve over time as a fixed portfolio.

For a given sequence of non-anticipative functions $g_0, \ldots, g_T$, and $\Pi_{0, T}^{c, H}(S) \equiv c+\sum_{k=0}^{T-1} H_k(S) \Delta_k S$, with $H_k$ non-anticipative and real numbers $a,b,c$ satisfying  $c \leq a < b$, define:
for any $S\in\Se$, set $\tau_0(S)\equiv 0$ and  define recursively,
\begin{equation} \nonumber
\rho_k(S) \equiv \inf\{n: \tau_k(S)\le n < T+1, g_n(S) \leq \Pi_{0, n}^{a, H}(S)\},\; \mbox{for}\; k\ge 0.
\end{equation}
\begin{equation}  \nonumber
\tau_k(S) \equiv \inf\{n: \rho_{k-1}(S)\le n < T+1; g_n(S) \geq \Pi_{0, n}^{b, H}(S)\},\; \mbox{for}\; k\ge 1.
\end{equation}
Moreover, as we have done before, we let these stopping times to equal $T+1$ if the sets where infimum is evaluated turn out to be empty.

\begin{corollary}[Dubins' inequality with portfolio bands, $T < \infty$] \label{dubinMovingBandsCorollary}
Let $g_0, \ldots, g_T$ be a given, real valued, supermartingale. Also, $\Pi_{0, k}^{c, H}$ is a portfolio process as indicated above and such that $g_j \geq  \Pi_{0, j}^{c, H}$ on $\mathcal{S}$ and for all $0 \leq j \leq T$.

 Assume the definitions of $\rho_k, \tau_k$ as above, in particular $c \leq a <b$ are given  and $T < \infty$. Furthermore assume that $(L)-a.e.$ holds, then for any $k \geq 0$:
\begin{equation} \label{manyIntervals}
\overline{\sigma}({\bf 1}_{\{\tau_{k+1} < T\}}) \leq
\Big(\frac{a-c}{b-c} \Big)^{k+1}.
\end{equation}
\end{corollary}
\begin{proof}
Define $f_j = g_j - \Pi_{0,j}^{0, H}$ for $0 \leq j \leq T$, this is a supermartingale sequence that satisfies $f_j \geq c$ on $\mathcal{S}$ and for all $0 \leq j \leq T$. Then, the stopping times
$\rho_k, \tau_k$ defined above count exactly the upcrosses of the sequence $\{f_j\}_{0 \leq j \leq T}$ through the band $[a,b]$. Then (\ref{manyIntervals}) follows from Theorem \ref{dubinFiniteTime}.
\end{proof}

In the above proof, the portfolio $H_k$ is completely arbitrary (besides being non-anticipative) showing how the upperbound depends on the initial investments $a,b$ (as well as $c$) only.

\subsection{Dubins' Inequality for Infinite Time}
We now extend the previous theorem by considering the case when $T \leq \infty$.
The main novelty is that infinite time portfolios, i.e. involving $\liminf$, are required for upperbounding and hence we need to make sure they are nonnegative, to enforce this property will require an additional hypothesis.  We also note that
the result is stated for $\overline{I}$ as opposed to $\overline{\sigma}$ and so providing some more generality but,  this additional generality is illusory as the hypotheses required for the $T= \infty$ case imply $\overline{\sigma} = \overline{I}$ on positive functions as we show in Section \ref{propertyKae}. For simplicity we will take $c=0$ in the formulation of the next theorem (compare to Theorem \ref{dubinFiniteTime}).

\vspace{.1in}
One could speculate that  a more natural approach to proving the result would be to cut-off time to a finite value, apply the previous result and then take an appropriate limit. This approach is available in the stochastic setting due to the monotone convergence theorem. Given that we are working with the
sublinear operator $\overline{\sigma}$ we do not have a result that would allow us to apply the operator first and then take the limit. 


\begin{theorem}[Dubins' inequality for infinite time, $T = \infty$] \label{dubinInfiniteTimeA}
Let $\{f_j\}_{j \geq 0}$ be a given real valued non-negative supermartingale sequence and $\rho_k, \tau_k$ given by Definition \ref{upcrossing_stopping_times}  with $T = \infty$ (i.e. the infinite time case). Furthermore, assume that $(L)-a.e.$ holds as well as property $(P)$ (as per Definition \ref{def:Positivity})  then,
for any $k\ge 0$ and $0 \leq a < b$:
\begin{equation} \nonumber 
\overline{I}({\bf 1}_{\{\tau_{k+1} < \infty\}}) \leq \frac{a}{b}~~ \overline{I}({\bf 1}_{\{\rho_{k}  < \infty\}}) \leq
\frac{a}{b}~~ \overline{I}({\bf 1}_{\{\tau_{k}  < \infty\}}),
\end{equation}
and so
\begin{equation} \label{couldBeCheckedEmpiricallyA}
\overline{I}({\bf 1}_{\{\tau_{k+1}  < \infty\}}) \leq
(\frac{a}{b})^{k+1} ~\overline{I}({\bf 1}_{\{\rho_{0}  < \infty\}}) \leq
(\frac{a}{b})^{k+1}.
\end{equation}
\end{theorem}
\begin{proof}
We will prove: $\overline{I}({\bf 1}_{\{\tau_{k+1} < \infty\}}) \leq
\frac{a}{b}~~ \overline{I}({\bf 1}_{\{\rho_{k} < \infty\}})$ and remark that
$\{\rho_{k} < \infty\} \subseteq
\{\tau_{k} < \infty\}$. These facts will complete the proof of the theorem.

Define, for a fixed $k \geq 0$: 
\begin{equation} \nonumber
f_j^{[k]} \equiv \overline{I}_j({\bf 1}_{\{\rho_k < \infty\}} f_{\rho_k})~\mbox{whenever}~0 \leq j \leq \rho_k   
\end{equation}
and 
\begin{equation} \nonumber
f_j^{[k]}= f_j~\mbox{whenever}~j > \rho_k.    
\end{equation} 
We argue next that $\{f_j^{[k]}\}_{ j \geq 0}$ is a supermartingale
(indeed, we are pasting two supermartingales) taking values on $[0, \infty)$.

Consider an arbitrary node $(S,j)$. In the case when  $j < \rho_{k}(S)$ we have $j+1\leq\rho_k(\tilde S)$ for every $\tilde S\in \Se_{(S,j)}$, and, hence, $f_{j+1}^{[k]}(\tilde S)=  \overline{I}_{j+1}({\bf 1}_{\{\rho_k < \infty\}} f_{\rho_k})(\tilde S)$. Therefore,	
\begin{equation} \nonumber
\overline{\sigma}_j f_{j+1}^{[k]}(S) \leq \overline{I}_j(\overline{I}_{j+1} {\bf 1}_{\{\rho_k < \infty\}} f_{\rho_k})(S) \leq  \overline{I}_j( {\bf 1}_{\{\rho_k < \infty\}} f_{\rho_k})(S)= f_j^{[k]}(S),
\end{equation}
where the first inequality follows from $\overline{\sigma}_j \leq \overline{I}_j$, which holds on $P$, and the last  inequality is the tower property for the operator $\overline{I}_j$ (namely Corollary \ref{towerPropertyForI-bar}). For the case case when $j\geq  \rho_k(S)$ we have $j\geq  \rho_k(\tilde S)$ for every $\tilde S\in \Se_{(S,j)}$ and, hence, $f_{j+1}^{[k]}(\tilde S)=  f_{j+1}(\tilde S)$. Thus,
\begin{equation} \nonumber
\overline{\sigma}_j f_{j+1}^{[k]}(S) = \overline{\sigma}_j(f_{j+1})(S) \leq f_j(S)
= f_j^{[k]}(S), \quad \textnormal{a.e.},
\end{equation}
where the  inequality is the supermartingale property of $\{f_j\}_{j\geq 0}$. The last equality in the previous display is obvious for $j>\rho_k(S)$, whereas on $\{\rho_k=j\}$, $f_j^{[k]} = \overline{I}_j({\bf 1}_{\{\rho_k < \infty\}} f_{\rho_k})=\overline{I}_jf_j=f_j$ a.e. thanks to assumption $(L)$-a.e. The above computations establish $\overline{\sigma}_j f_{j+1}^{[k]} \leq f_j^{[k]}$ a.e.

$f_j^{[k]}(S) \geq 0$ for all $S$ is clear given that $\overline{I}_j$ is isotone and $\overline{I}_j0(S)=0$ for all $(S,j)$). The fact that $f_j^{[k]}(S) < \infty $  follows from $\overline{I}_j({\bf 1}_{\{\rho_k < \infty\}} f_{\rho_k(S)})(S) \leq a$ and $f_j(S) < \infty$ both valid for all $S$.

\vspace{.1in}
From the above considerations, Theorem \ref{thm:decomposition}  is applicable and yields the validity of the following identity a.e.:
\begin{equation}  \label{usefulT-Infinite}
   f^{[k]}_{\tau_{k+1} \wedge n}(S) =   f^{[k]}_{0}(S)+ \sum_{j=0}^{\tau_{k+1} \wedge n -1} H_j(S) ~\Delta_j S - A_{\tau_{k+1} \wedge n}+ \sum_{j=0}^{\tau_{k+1} \wedge n -1} \delta_j
\end{equation}
with $\delta_j >0$ are given constants and summable ($\delta \equiv \sum_{j \geq 0}\delta_j$) but, otherwise arbitrary, the $A_j(S)$ are non-anticipative non-decreasing functions with $A_0=0$. Of course, the functions $H_j$ are non-anticipative as well.
From (\ref{usefulT-Infinite}), the fact that $f_{j}^{[k]} \geq 0$  and passing to $\liminf_{n \rightarrow \infty}$, it follows that the following holds everywhere on $\mathcal{S}$
\begin{equation}  \label{fromUsefulT-Infinite}
   {\bf 1}_{\{\tau_{k+1}  < \infty\}}~f^{[k]}_{\tau_{k+1}}(S)  \leq   f^{[k]}_{0}(S)+ \liminf_{n \rightarrow \infty} \sum_{j=0}^{n -1} F_j(S) ~\Delta_j S + \delta +\infty ~ {\bf 1}_{\mathcal{M}},
\end{equation}
where $F_i \equiv {\bf 1}_{\{i < \tau_{k+1}\}}~ H_i$ and ${\mathcal{M}}$ is a null set. From Lemma 7.3 in \cite{bender3} (which is available thanks to our assumption $(P)$) we know that
\begin{equation} \nonumber
    f_0^{[k]} + \delta + \sum_{j=0}^{i-1} H_j(S) ~\Delta_j S\geq 0
\end{equation}
holds for all $S$ and $i \geq 0$. Therefore: $ 0 \leq (\sum_{j=0}^{(i-1)\wedge (\tau_{k+1} -1)} H_j(S) ~\Delta_j S + \delta +f_0^{[k]} ) =  \sum_{j=0}^{i-1} F_j(S) ~\Delta_j S + \delta+ f_0^{[k]}$. This later property allows to apply $\overline{I}$ to (\ref{fromUsefulT-Infinite}), leading to
\begin{equation}  \nonumber
  \overline{I}({\bf 1}_{\{\tau_{k+1}  < \infty\}}~f^{[k]}_{\tau_{k+1}}) \leq   f^{[k]}_{0}+ \delta= \overline{I}({\bf 1}_{\{\rho_k < \infty\}} f_{\rho_k}) + \delta \leq a ~\overline{I}({\bf 1}_{\{\rho_k < \infty\}})+\delta,
\end{equation}
but $f^{[k]}_{\tau_{k+1}}= {\bf 1}_{\{\rho_k < \infty\}} f_{\tau_{k+1}}$ and ${\bf 1}_{\{\tau_{k+1}< \infty\}}~{\bf 1}_{\{\rho_k < \infty\}}= {\bf 1}_{\{\tau_{k+1} < \infty\}}$. Therefore from the above inequality we get
\begin{equation}  \nonumber
  \overline{I}({\bf 1}_{\{\tau_{k+1}  < \infty\}}) \leq  \frac{a}{b} ~\overline{I}({\bf 1}_{\{\rho_k < \infty\}})+ \frac{\delta}{b}.
\end{equation}
Given that $\delta$ is arbitrary we have completed the proof.
\end{proof}

\section{Property $(K)-a.e.$}  \label{propertyKae}

Consider a node $(S,j)$, for any $f \geq 0$, i.e. $f \in P$,  we have 
$\overline{\sigma}_j f \leq \overline{I}_j f$ and if $(L_{(S,j)})$ holds we then have
$0 \leq \overline{\sigma}_j f \leq \overline{I}_j f$. It is possible, for some $f$ and some trajectory sets, to have
$\overline{\sigma}_j f < \overline{I}_j f$ (see Example 1 in \cite{bender3}).
The possibility that $\overline{\sigma}_j \neq \overline{I}_j$ on positive functions reflects
the generality of the results we have presented, in particular, this feature of our theory set us apart from game theoretical probability as presented in 
\cite{shafer} (see the detailed comparison in Section 4 of \cite{bender3}).

In this section we provide  sufficient conditions to guarantee $\overline{\sigma}_j f = \overline{I}_j f$ a.e. on non-negative functions and, at the same time, shed some light on the phenomena. In particular, we will have sufficient conditions for the validity of $\overline{\sigma} f= \overline{I}f$ for all $f \in P$, a fact that makes several of the results in the paper available to the operator $\overline{I}$.

The property $(L)$-a.e.\ plays a central role throughout the paper. 
The stronger property $(K)$-a.e., introduced below, complements several of our main results by providing a structural criterion for the equality $\overline{\sigma}_j f = \overline{I}_j f$ to hold a.e. on non-negative functions. 

As with property $(L_{(S,j)})$, it is convenient to formulate the condition locally at a fixed node $(S,j)$. The following definition isolates the corresponding nodewise property $(K_{(S,j)})$, which will later be shown equivalent to the equality
\[
\overline{\sigma}_j f(S)=\overline{I}_j f(S), \qquad f\in P.
\]

\begin{definition}[$(K_{(S,j)})$]  \label{definitionOfK}
Fix an arbitrary node $(S,j)$ and let  $g_n \equiv \sum_{m=0}^{n} f_m$  with $f_0 \in \mathcal{E}_{(S,j)}$ and 
$f_m = \liminf_{k \rightarrow \infty} \Pi^{V^m, H^m}_{j, k}$ with
$\Pi^{V^m, H^m}_{j, k} \in \mathcal{E}^+_{(S,j)}$ for all $k \geq j$ and $m \geq 1$. We say that $(K_{(S,j)})$ holds if the following inequality is valid
for any such function $g_n$ and any $n \geq 0$
\end{definition}

\vspace{-.1in}
\begin{equation}  \label{konigInequality}
\overline{I}_j g_n^+(S) \leq V^0(S)+ \ldots+ V^n(S)+ \overline{I}_jg_n^-(S).
\end{equation}
\begin{remark}
(\ref{konigInequality}) represents a version of the same named  property defined in \cite{bender1} (but originally introduced in \cite{konig}) but here adapted to our use of $\liminf$. In a setting without $liminf$ the defining inequality is equivalent to being an equality  (\cite{bender1})).
\end{remark}

\begin{definition}[$(K)-a.e.$] \label{definitionOfKa.e.}
For a fixed $j \geq 0$, we will write $(K_j)-a.e$ whenever $(K_{(S,j)})$ holds a.e. in $S$. We will also write $(K)-a.e$ whenever $(K_j)-a.e$ holds for all $j \geq 0$.
\end{definition}
Note that the case $j=0$ is somewhat special as there are only two cases, namely $(K_{(S,0)})$ holds for all $S$ or for no $S$. If $(L)$ is in force, then $\Se$ is a full set and, therefore, $(K)-a.e$ implies that $(K_{(S,0)})$ holds at the initial node.


\noindent
Our interest in $(K_{(S,j)})$ stems from the following proposition.
\begin{proposition} \label{konigAtFixedNode}
Let $(S,j)$ be an arbitrary node. Then:
\begin{enumerate}
\item If $(K_{(S,j)})$ holds then $(L_{(S,j)})$ holds.
\item $ (K_{(S,j)}) ~\mbox{holds if and only if}~ \overline{I}_j f(S) = \overline{\sigma}_j f(S)~\mbox{for any}~f \in P$.
\end{enumerate}
\end{proposition}
\begin{proof}

Consider $0 \leq \sum_{m \geq 0} f_m$ on $\mathcal{S}_{(S,j)}$ where the $f_m$ are as in Definition \ref{definitionOfK}. It is enough to establish
$0 \leq \sum_{m \geq 0} V^m(S)$, therefore, we may assume $\sum_{m \geq 0} V^m(S) < \infty$.   

It follows that 
\begin{equation} \nonumber
0 \leq \overline{I}_j g_n^+(S) + \overline{I}_j(\sum_{m \geq n+1} f_m)(S) \leq V^0(S)+ \ldots+ V^n(S) +\overline{I}_j g_n^-(S)+ \overline{I}_j(\sum_{m \geq n+1} f_m)(S). 
\end{equation}
We also note that $g_n^- =(-g_n)^+ = (-\sum_{m=0}^n f_m)^+ \leq \sum_{m \geq n+1} f_m$ and so $\overline{I}_j g_n^-(S) \leq  \overline{I}_j (\sum_{m \geq n+1} f_m)(S)$.
Therefore 
\begin{equation} \nonumber
0 \leq V^0(S)+ \ldots+ V^n(S) + 2~\overline{I}_j (\sum_{m \geq n+1} f_m)(S). 
\end{equation}
Notice that $0 \leq \overline{I}_j (\sum_{m \geq n+1} f_m)(S) \leq \sum_{m \geq n+1} V^m(S)$ and the latter term converges to $0$ as $n \rightarrow \infty$ due to $\sum_{m  \geq 0} V^m(S) < \infty$. Therefore, $\lim_{n \rightarrow \infty} \overline{I}_j (\sum_{m \geq n+1} f_m)(S)=0$ it then follows that $0 \leq \sum_{m \geq 0} V^m(S)$. Hence, $\overline \sigma_j 0(S)\geq 0$, which is equivalent to the validity of $(L_{(S,j)})$ by Proposition 3.3 in \cite{bender3}.

To establish the direct implication in  item $(2)$ it is enough to prove $\overline{I}_j f(S) \leq \overline{\sigma}_j f(S)$, so we may concentrate in the case when $\overline{\sigma}_j f(S) < \infty$ holds as well. 
Consider $f_m$ as in Definition \ref{definitionOfK}  satisfying $f \leq \sum_{m \geq 0} f_m$ on $\mathcal{S}_{(S,j)}$. Therefore,  $f \leq g_n^++ \sum_{m \geq n+1} f_m$ on $\mathcal{S}_{(S,j)}$ where $g_n \equiv \sum_{m =0}^{n} f_m$. It then follows that
\begin{equation} \nonumber
\overline{I}_j f(S) \leq \overline{I}_j g_n^+(S) + \overline{I}_j (\sum_{m \geq n+1} f_m)(S) \leq V^0(S)+ \ldots+ V^n(S) + \overline{I}_j g_n^-(S) + \overline{I}_j (\sum_{m \geq n+1} f_m)(S). 
\end{equation}
As we argued above
\begin{equation} \nonumber
\overline{I}_j f(S) \leq V^0(S)+ \ldots+ V^n(S) + 2~\overline{I}_j (\sum_{m \geq n+1} f_m)(S),
\end{equation}
moreover, the last term in the above display converges to $0$ as $n \rightarrow \infty$ and so  $\overline{I}_j f(S) \leq \sum_{m \geq 0} V^m(S)$. It then follows that $\overline{I}_j f(S) \leq \overline{\sigma}_j f(S)$.

It remains to establish the converse implication in  item $(2)$ but,
$g_n^+ = g_n +g_n^-$ with $g_n = \sum_{m=0}^n f_m$ and $f_m$ as introduced in Definition \ref{definitionOfK}. Therefore $\overline{\sigma}_j g_n^+(S) \leq \overline{\sigma}_j(\sum_{m=0}^n f_m)(S)+\overline{\sigma}_j g_n^-(S) \leq \sum_{m=0}^n V^m(S)+\overline{\sigma}_j g_n^-(S)$.  Since $g_n^+,g_n^- \in P$, the assumed equality $\overline{I}_j=\overline{\sigma}_j$ on $P$ implies (\ref{konigInequality}), completing the proof. 
\end{proof}
The following result follows from Proposition \ref{konigAtFixedNode}.
\begin{corollary}
\begin{equation} \nonumber
(K)-a.e. ~\mbox{holds if and only if}~ \overline{I}_jf = \overline{\sigma}_jf~\mbox{holds a.e. for each}~ j \geq 0 ~\mbox{and for all}~ f \in P.
\end{equation}
\end{corollary}

The above results provide a sense for the meaning of $(K)-a.e.$; we will establish this property by requiring some weak sufficient conditions introduced next. We first need some definitions from \cite{bender3}[Section 3]. 

\begin{definition}[Types of nodes] \label{typesOfNodes} 
Given a trajectory space $\Se$ and a node $(S,j)$:
	\begin{itemize}
		\item $(S,j)$ is called an \emph{up-down node} if
		\begin{equation} \label{upDownProperty}
			\sup_{\tilde{S} \in \Se_{(S, j)}}~~ (\tilde{S}_{j+1} - S_{j}) >    0\quad \mbox{and}\quad
			\inf_{\tilde{S} \in \Se_{(S, j)}} ~~(\tilde{S}_{j+1} - S_{j}) <    0.
		\end{equation}
		\item $(S,j)$ is called a \emph{flat node} if
		\begin{equation} \label{flat}
			\sup_{\tilde{S} \in \Se_{(S, j)}}~~ (\tilde{S}_{j+1} - S_{j}) =    0 =
			\inf_{\tilde{S} \in \Se_{(S, j)}} ~~(\tilde{S}_{j+1} - S_{j}).
		\end{equation}
	\end{itemize}
	\noindent $(S, j)$ is called an arbitrage-free node if (\ref{upDownProperty}) or (\ref{flat}) hold, otherwise it
	is called an \emph{arbitrage node}. An arbitrage node $(S, j)$ is said to be of \emph{type I}, if there exists $\hat{S}\in\Se_{(S, j)}$	such that $\hat{S}_{j+1} = S_j$; otherwise it is said to be of \emph{type II}.
\end{definition}

\begin{definition}[Bad nodes]
	For a fixed node $(S,j)$, let
	\begin{equation}\label{eq:set_bad}
		N(S,j)=\{\tilde S\in \Se_{(S,j)}:\; (\tilde S,k) \mbox{ is arbitrage node and } \tilde S_{k+1}\neq \tilde S_k \mbox{ for some } k\geq j\}.
	\end{equation} 
	A node $(S,j)$ is called \emph{bad}, if $\mathcal{S}_{(S,j)}= N(S,j)$. Otherwise, $(S,j)$ is said to be \emph{good}.
\end{definition}

\begin{definition}[Trajectorial completeness]\label{def:TC}
	Suppose $(S^n)_{n\ge 0}$ is a sequence in $\mathcal{S}$ satisfying
	\begin{equation}\label{splittingSequence}
	S^n_i = S^{n+1}_i, \;\; 0 \leq i \leq n,
	\end{equation}
	for all $n\in \mathbb{N}_0$. Then, its \emph{limit} is defined as
	$$
	\lim_{n \rightarrow \infty} S^n\equiv \overline{S}\equiv (\overline{S}_i)_{i \geq 0},~~\mbox{wherein}~~~\overline{S}_i \equiv S^i_i.
	$$
	Denote by $\overline{\Se}$ be the set of all such limits $\overline{S}$. Then, $\overline{\Se}$ is called the \emph{trajectorial completion} of $\Se$ and  the trajectory set $\Se$ is said to be  \emph{trajectorially complete}, (TC) for short, if $\Se=\overline{\Se}$.
\end{definition}

The following definition introduces two conditions, (TC$_{\mbox{bad}}$) and (P$_{\mbox{bad}}$),  that we will require in order to prove $(K)-a.e.$. \cite{bender3}[Proposition 7.2] shows them to be sufficient conditions for Doob's pointwise supermartingale convergence theorem (Theorem 7.1 in the said reference).
 
\begin{definition}[Sufficient conditions for $(K)-a.e.$] 
The following two properties are relative to $\mathcal{S}$.
\begin{itemize}[labelwidth=!, leftmargin=*, align=left]
\item[(TC$_{\mbox{bad}}$)] If  $(S^n)_{n\ge 0}$  is a sequence in $\mathcal{S}$ satisfying \eqref{splittingSequence}
and if there is an $n_0\geq 0$ such that $(S^n,n)$ is a good node for every  $n\geq n_0$, then $\lim_{n \rightarrow \infty} S^n~\in \mathcal{S}$.\\

\item[(P$_{\mbox{bad}}$)]  Whenever $(S,k)$ is a good up-down node such that $(S,k+1)$ is bad, then there are $S^1, S^2\in \Se_{(S,k)}$ such that $S^1_{k+1}> S_{k+1}> S^2_{k+1}$ and $(S^n,j+1)$ are good for $n =1,2$.\\

\end{itemize}
\end{definition}

\begin{definition}[$(L)-\overline{I}_j~a.e.$]
For a given node $(S,j)$, the following two properties will be referred  as the property $(L)-\overline{I}_j~a.e.$ on $\mathcal{S}_{(S,j)}$,
 \begin{enumerate}
     \item $(L_{(S,j)})$ holds,
     \item $\mathcal{N}^{(L_{(S,j)})} \equiv \{\hat{S} \in \mathcal{S}_{(S,j)}: \exists~k~\geq j~ \mbox{s.t.}~ (L_{(\hat{S}, k)}) ~\mbox{fails}\}$~is~a ~$\overline{I}_j$-null set (in particular $(L_{(\hat{S}, k)})$~holds~$\overline{I}_j-a.e.$~ for~ every~ $k \geq j$). 
  \end{enumerate}
\end{definition}

\begin{proposition} \label{prop:locaL-ae}
Fix  $j$, $j \geq 0$, and assume that  $(S,j)$ is  a good node. Moreover, assume that  (TC$_{\mbox{bad}}$) and  (P$_{\mbox{bad}}$) both hold. Then, $(L)-\overline{I}_j~a.e.$ holds; also,  property $(P)$ holds (as per Definition \ref{def:Positivity}) when restricted to $\mathcal{S}_{(S,j)}$
\end{proposition}
\begin{proof}
We note that (TC$_{\mbox{bad}}$) and (P$_{\mbox{bad}}$) are properties that hold on all $\mathcal{S}$ but they also hold, by their form, when restricted to an arbitrary conditional set $\mathcal{S}_{(S,j)}$. With this comment in mind, a proof of our proposition follows exactly the arguments in Proposition 7.2 of \cite{bender3} but now deployed on the trajectory set $\mathcal{S}_{(S,j)}$ (as opposed to the set $\mathcal{S}$). Note that the assumption that $(S,j)$ is a good node is needed for invoking  Corollary 3.3 in \cite{bender3} with $\Se_{(S,j)}$ in place of $\Se$.

\end{proof}

\begin{theorem} \label{nodewiseImplicationOfK}
Assume the same hypotheses as in Proposition \ref{prop:locaL-ae} then,
\[
\overline I_j f(S)=\overline\sigma_j f(S), \qquad f\in P.
\]
Consequently, by Proposition~\ref{konigAtFixedNode}, property $(K_{(S,j)})$
holds.
\end{theorem}
\begin{proof}
We only need to establish $\overline{I}_j f(S) \leq \overline{\sigma}_j f(S)$ for an arbitrary $f \in P$, we may then assume $\overline{\sigma}_j f(S) < \infty$.

Fix $f \in P$ and consider the following inequality valid on $\mathcal{S}_{(S,j)}$ 
\begin{equation} \nonumber
f \leq \Pi_{j, n_0}^{V^0,H^0}+ \sum_{m \geq 1} \liminf_{n \rightarrow \infty} \Pi_{j, n}^{V^m, H^m},    
\end{equation}
where $\Pi_{j, n_0}^{V^0,H^0}\!\!\in \mathcal{E}_j$ and, for every  $m\geq 1$, $\Pi_{j, n}^{V^m, H^m}\!\!\in \mathcal{E}^+_j$ for every $n\geq j$ and $\sum_{m=1}^\infty V^m(S)<\infty$.

Given an arbitrary  $\hat{S} \in \mathcal{S}_{(S,j)}$, set the notation  $g_n^m(\hat{S}) \equiv V^m(S)+ \sum_{k=j}^{n-1} H_k^m(\hat{S}) \Delta_k \hat{S}$, for all $m \geq 0, ~n \geq j$; notice that, whenever $ m \geq 1$,  $g_n^m(\hat{S}) \geq 0$  for all $n \geq j$ and define
\begin{equation} \nonumber
    g_n(\hat{S})= \sum_{m\geq 1} g^m_n(\hat{S}). 
\end{equation}
It follows from Fatou's lemma that the following inequality holds on $\mathcal{S}_{(S,j)}$
    \begin{equation} \label{aggregationGivesASupermartingalea.e.}
    f \leq \Pi_{j, n_0}^{V^0,H^0}+ \liminf_{n \rightarrow \infty}~\sum_{m \geq 1} g^m_n= \Pi_{j, n_0}^{V^0,H^0} + \liminf_{n \rightarrow \infty}~g_n=  \liminf_{n \rightarrow \infty}~(g_n^0+ g_n).
    \end{equation}

    Taking $n \geq j$, we observe that, for any $\hat{S} \in \mathcal{S}_{(S,j)}$,
    \begin{equation} \nonumber
        f(\hat{S}) \leq g_n^0(\hat{S}) +\sum_{k=n}^{n_0-1} H_k^0(\hat{S}) \Delta_k \hat{S}+ \sum_{m \geq 1}[g_n^{m}(\hat{S}) + \liminf_{p \rightarrow \infty} \sum_{k=n}^{p-1} H_k^m(\hat{S}) \Delta_k \hat{S}],
    \end{equation}
    from which it follows that for any $\hat{S} \in \mathcal{S}_{(S,j)}$ and any $n \geq j$
    \begin{equation}  \label{givesPositivityOnlya.e.}
        \overline{\sigma}_n f (\hat{S}) \leq g_n^0(\hat{S})+ g_n(\hat{S}).
    \end{equation}
We will consider the following stopping time defined on $\mathcal{S}_{(S,j)}$, for each $\hat{S} \in \mathcal{S}_{(S,j)}$ define:
    \begin{eqnarray*}
	\tau^{\#}(\hat{S})= \inf \{k\ge j:& (L_{(\hat{S},k)})\; \mbox{fails, or}\; [(\hat{S},k-1)\; \mbox{is a type I arbitrage node (see Definition \ref{typesOfNodes}) and}\; \hat{S}_k\neq \hat{S}_{k-1}]\}
	\end{eqnarray*}
where the second alternative is only relevant for $k>j.$ 	
Define, for $n \geq j$,
\begin{equation} \label{stoppedDefinition}
    \tilde{g}_n= (g_n^0+ g_n)~{\bf 1}_{\{n< \tau^{\#} \}},
\end{equation}
we argue next that $\{  \tilde{g}_n\}_{n \geq j }$ is a $[0, \infty)$ valued supermartingale. We note, in passing, that the need to consider $\tilde{g}_n$, instead of $g^0_n+g_n$, emerges as the latter could take the value $\infty$ while the former takes finite values as we show next (finite values are required whenever relying on the supermartingale decomposition, as we do below).

To check $ \tilde{g}_n\geq 0$ on $\mathcal{S}_{(S,j)}$ we need only consider the case $n < \tau^{\#}(\hat{S})$ but this implies that $(L_{(\hat{S},n)})$ holds and so $\overline{\sigma}_n f(\hat{S}) \geq 0$ and then $ \tilde{g}_n(\hat{S}) \geq 0$ follows from (\ref{givesPositivityOnlya.e.}).

Next, we establish that $\tilde{g}_n< \infty $ on $\mathcal{S}_{(S,j)}$. We need to check that $\sum_{m=0}^{\infty} \Pi_{j, n}^{V^m, H^m}(\hat{S})$  is finite if $n<\tau^{\#}(\hat{S})$. In this case, for every $j \leq k <n$, the node $(\hat{S},k)$ is an up-down node or $\hat{S}_{k+1}= \hat{S}_k$. Indeed, if $(\hat{S},k)$ were an arbitrage node of type II, then, by Proposition 3.10 from (\cite{bender3}), $(L_{(\hat{S},k)})$ fails, which results in $\tau^{\#}(\hat{S})\leq k < n$; a contradiction. Similarly, we arrive at a contradiction, if  $(\hat{S},k)$ were an arbitrage node of type I and $\hat{S}_{k+1}\neq \hat{S}_k$, because, then, $\tau^{\#}(\hat{S})\leq k+1\leq n$. Thus, we may apply the Aggregation Lemma 6.4 from \cite{bender3} to conclude that, for every $j\leq k < n$, the series $\sum_{m=0}^{\infty} (H^m_n(\hat{S})\Delta_k \hat{S})$,   converges in $\mathbb{R}$. Consequently
$$
\sum_{m=0}^\infty \Pi_{j, n}^{V^m, H^m}(\hat{S})=\sum_{m=0}^\infty V^m(S) +\sum_{k=j}^{n-1} \sum_{m=0}^\infty (H^m_k(\hat{S})\Delta_k \hat{S}) \in \mathbb{R}.
$$

\vspace{.1in}
We will now establish that $\{\tilde{g}\}_{n \geq j}$ is a supermartingale on $\mathcal{S}_{(S,j)}$, i.e. we will prove that $\overline{\sigma}_n \tilde{g}_{n+1} \leq \tilde{g}_n$ $\overline{I}_j-a.e.$ for all $ n\geq j$. Let $B$ denote the set of trajectories $\hat{S} \in \mathcal{S}_{(S,j)}$  such that  $g_n(\hat{S}) < \infty$; on $B$ we have 
\begin{equation}  \label{needGjFinite}
    g_{n+1}(\hat{S})= g_{n}(\hat{S})   + \sum_{m \geq 1} H^m_n(\hat{S}) \Delta_n \hat{S},
\end{equation}
with, possibly, the last term being equal to $\infty$ (and so, also the left hand side). For all $\hat{S} \in B$ we obtain
\begin{equation}  \label{stillWithInfinities}
    \tilde{g}_{n+1}(\hat{S} ) \leq \tilde{g}_{n}(\hat{S} ) + H_n(\hat{S}) \Delta_n \hat{S}  +\infty ~{\bf 1}_{\{\tau^{\#} < \infty \}}(\hat{S} )
\end{equation}
where $H_n\equiv H^0_n+\sum_{m \geq 1} H^m_n$ on $\{n< \tau^{\#} \}$  and $H_n \equiv 0$ otherwise. The series defining $H_n$ indeed converges, this fact follows the same argument (provided above) that showed that $g_n(\hat{S} ) < \infty$ whenever $n< \tau^{\#}(\hat{S})$. It also follows that $H_n$ is non-anticipative.
To check (\ref{stillWithInfinities}) for $\hat{S}  \in B$, we use: ${\bf 1}_{\{n+1 < \tau^{\#} \}}(\hat{S})= {\bf 1}_{\{n< \tau^{\#} \}}(\hat{S})- {\bf 1}_{\{n+1 = \tau^{\#} \}}(\hat{S})$
and (\ref{needGjFinite}) to expand $ \tilde{g}_{n+1}(\hat{S})$ in (\ref{stoppedDefinition}). The terms, resulting from the expansion and containing ${\bf 1}_{\{n+1 = \tau^{\#} \}}$ are absorbed in the term
$\infty ~{\bf 1}_{\{\tau^{\#} < \infty \}}$ given that
$\{n+1 = \tau^{\#} \} \subseteq \{\tau^{\#} < \infty \}$
where the latter is a null set (see below).

We now argue that (\ref{stillWithInfinities}) actually holds in all of $\mathcal{S}_{(S,j)}$, towards this end consider $\hat{S} \in \mathcal{S}_{(S,j)} \setminus B$. Assuming  that ${\bf 1}_{\{\tau^{\#} > n \}}(\hat{S})=1$, we can conclude that  $\tilde{g}_{n}(\hat{S}) = \infty$, which contradicts the finiteness of $\tilde g_n$. If, however, 
${\bf 1}_{\{\tau^{\#} \leq  n \}}(\hat{S})=1$ we get ${\bf 1}_{\{\tau^{\#} < \infty \}}(\hat{S})=1$, and  (\ref{stillWithInfinities}) trivially holds.

$\{\tau^{\#} < \infty \}$ is a $\overline{I}_j$-null subset of $\mathcal{S}_{(S,j)}$, this follows from Lemma A.1.3 from \cite{bender3} and by the validity of $(L)-\overline{I}_j~a.e.$ (which is available by Proposition \ref{prop:locaL-ae}). Therefore,
from (\ref{stillWithInfinities}), being valid on $\mathcal{S}_{(S,j)}$, and after acting with $\overline{\sigma}_n$ we obtain $\overline{\sigma}_n \tilde{g}_{n+1} \leq \tilde{g}_n+ \overline{I}_n( \infty ~{\bf 1}_{\{\tau^{\#} < \infty \}})$. From the tower property  in Corollary \ref{towerPropertyForI-bar}  and the fact that $\{\tau^{\#} < \infty \}$ is $\overline{I}_j$-null we obtain 
$\overline{I}_j(\overline{I}_n(\infty ~{\bf 1}_{\{\tau^{\#} < \infty \}})) \leq \overline{I}_j(\infty ~{\bf 1}_{\{\tau^{\#} < \infty \}})=0$ and so $\overline{I}_n(\infty~{\bf 1}_{\{\tau^{\#} < \infty \}})=0 ~\overline{I}_j-a.e.$ It then follows that $\overline{\sigma}_n \tilde{g}_{n+1} \leq \tilde{g}_n$ $\overline{I}_j-a.e.$ on $\mathcal{S}_{(S,j)}$ for all $n \geq j$.

Also,  $\{\tau^{\#} < \infty \}$ being a $\overline{I}_j$-null subset of $\mathcal{S}_{(S, j)}$ shows that 
\begin{equation} \label{almostEqual} 
\tilde{g}_n = (g_n^0+g_n),~ \overline{I}_j-a.e.~\mbox{on}~\mathcal{S}_{(S, j)}
\end{equation}
 (and so the latter is a supermartingale on $\mathcal{S}_{(S, j)}$ as well).  Therefore, from (\ref{aggregationGivesASupermartingalea.e.}) and (\ref{almostEqual}),  we obtain that the following inequality is valid on $\mathcal{S}_{(S,j)}$
\begin{equation} \label{singleSupermartingaleBound}
f \leq   \liminf_{n \rightarrow \infty}~\tilde{g}_n+ \infty~{\bf 1}_{\mathcal{M}},
\end{equation}
where  $\mathcal{M}$ is a $\overline{I}_j$-null set.
We now apply the supermartingale decomposition Theorem \ref{thm:decomposition}
to the supermartingale $\{\tilde{g}_n\}_{n \geq j}$ on $\mathcal{S}_{(S,j)}$ (i.e. $\mathcal{S}_{(S,j)}$ is taking over the role of $\mathcal{S}$ in the said theorem). Theorem \ref{thm:decomposition} is applicable given that $(L)-\overline{I}_j~a.e.$ (which is available by Proposition \ref{prop:locaL-ae})
and the functions $\{\tilde{g}_n\}_{n \geq j}$ are real-valued.
The supermartingale decomposition in  conjunction with (\ref{singleSupermartingaleBound}), gives that the following inequality is valid on $\mathcal{S}_{(S,j)}$
\begin{equation} \label{backToMartingaleBound}
f \leq \tilde{g}_j+  \liminf_{n \rightarrow \infty}~\Pi^{0, H}_{j, n}+ \delta+\infty~{\bf 1}_{\tilde{\mathcal{M}}},
\end{equation}
where  $\tilde{\mathcal{M}}$ is a $\overline{I}_j$-null set.
Property $(P)$  on $\mathcal{S}_{(S,j)}$ (which is available by Proposition \ref{prop:locaL-ae}) and Lemma 7.3 from \cite{bender3} imply 
$\tilde{g}_j +\delta+ \sum_{k=j}^{n-1} H_k(\hat{S}) ~\Delta_k \hat{S} \geq 0~~\mbox{for all}~~\hat{S} \in \mathcal{S}_{(S,j)} ~\mbox{and}~ n \geq j$. We can then apply $\overline{I}_j$ to (\ref{backToMartingaleBound}) and obtain
\begin{equation} \label{almostThere2}
 \overline{I}_j f(S) \leq \tilde{g}_j(S) + \delta,   
\end{equation}
note that $\tilde{g}_j(S)  = (V^0(S) + g_j)~ {\bf 1}_{\{j < \tau^{\#}\}}= V^0(S) + \sum_{m \geq 1} V^m(S)$ given that $(L_{(S,j)})$ holds (the assumption that $(S,j)$ is a good node implies that $(L_{(S,j)})$ holds, this follows from Corollary 3.13 from \cite{bender3}, the applicability of the said corollary follows from our hypotheses as argued in  Proposition 7.2 of \cite{bender3})
and so $\{j < \tau^{\#}\} = \mathcal{S}_{(S,j)}$. But, $V^0 + \sum_{m \geq 1} V^m$  could have been taken arbitrarily close to $\overline{\sigma}_jf(S)$. This remark jointly with (\ref{almostThere2}) and the fact that $\delta$ is arbitrary completes the proof. 
\end{proof}

\begin{corollary}[Validity of $(K)-a.e.$]
Assume that $(TC_{bad})$ and $(P_{bad})$  hold and that $(S,0)$ is a good node. Then $(K)-a.e.$ holds  
\end{corollary}
\begin{proof}
From Corollary 3.3 and Proposition 7.2, item 1, in \cite{bender3} it follows that both $(L)-a.e.$ and property $(P)$ both hold.  Moreover, Corollary 3.13 in \cite{bender3} yields that $(L_{(S,j)})$ holds if and only if $(S,j)$ is a good node.  Therefore the 
hypothesis of $(S,j)$ being a good node is available a.e. for all $j \geq 0$ in our Proposition \ref{prop:locaL-ae}. This fact implies that the hypothesis in Theorem \ref{nodewiseImplicationOfK} are also available a.e. for all $j \geq 0$, it then follows from the said theorem that $(K)-a.e.$ holds.
\end{proof}

\backmatter

\bmhead{Acknowledgements}

The authors gratefully acknowledge the contributions of
Dr. Alfredo Gonz\'alez, who participated in the early stages
of this project and whose ideas influenced its development.
He passed away before the completion of this work.

The authors also thank Dr. Konrad Gajewski for contributions
made during the course of his Ph.D. thesis (\cite{konrad}), completed
in September 2022, where preliminary versions of some of
the results of this paper appeared.


\noindent
\textbf{Funding:}
S. Ferrando acknowledges partial support from an NSERC Discovery Grant.

























\begin{thebibliography}{99}

\bibitem{bender1} C. Bender, S.E. Ferrando and A.L. Gonzalez (2025),  \emph{Conditional Non-Lattice Integration, Pricing and Superhedging}. Revista de la Union Matematica Argentina, Volume {\bf 68}, No 2, 627-676.

%

\bibitem{bender3}  C. Bender, S.E. Ferrando, A.L. Gonzalez and K. Gajewski (2025),
\emph{Superhedging Supermartingales}.  International Journal of Approximate Reasoning, Volume 187, December 2025, 109567.




%
%

%




\bibitem{ferrando} S.E. Ferrando and A.L. Gonzalez (2018), \emph{Trajectorial martingale transforms.
Convergence and integration}. New York Journal of Mathematics, {\bf 24}, 702-738.



\bibitem{konrad} K. Gajewski (2022),
\emph{Non-Probabilisitic Supermartingales. Trajectorial Models for Two Stocks.}.
PhD. Thesis, Department of Mathematics, Toronto Metropolitan University, Toronto, ON, Canada, 2022. 
\url{https://doi.org/10.32920/29901713}







\bibitem{konig} H. K\"{o}nig (1982). \emph{Integraltheorie ohne Verbandspostulat}. Mathematische Annalen {\bf 258}, 447-458.


\bibitem
{leinert} M. Leinert (1982),  \emph{Daniell-Stone integration without the lattice condition}. Archiv der Mathematik, {\bf 38}, 258-265.

%
%
%
%
%
%
\bibitem{shafer} G. Shafer and V. Vovk (2019),
Game-Theoretic Foundations for Probability and Finance. Wiley.

\bibitem{shiryaev} A.N. Shiryaev (2007),
Optimal Stopping Rules. Springer Verlag.


\bibitem{VdVW}
A. W. van der Vaart and J. A. Wellner (1996), Weak Convergence and Empirical Processes. Springer Verlag.

%


\end{thebibliography}
\end{document}